\documentclass[pdf, 10pt]{amsart}

\usepackage{amsthm, amsmath, amssymb}
\usepackage[OT1]{fontenc}
\usepackage{color}
\usepackage{dsfont}

\numberwithin{equation}{section}

\newtheorem{thm}[equation]{Theorem}

\newtheorem{lem}[equation]{Lemma}

\def\R{{\mathbb{R}}}

\def\<{\langle}
\def\>{\rangle}
\def\Q{{\bf Q}}
\def\P{{\bf P}}
\def\S{{\bf S}}
\def\T{{\bf T}}

\DeclareMathAlphabet{\mathpzc}{OT1}{pzc}{m}{it}

\renewcommand{\atop}[2]{\substack{{#1}\\{#2}}}

\begin{document}

\title[A multilinear local $T(b)$ theorem] {A local $T(b)$ theorem for perfect multilinear Calder\'{o}n--Zygmund operators}

\author{Mariusz Mirek}
\address{M. Mirek \\
	Universit\"{a}t Bonn \\
	Mathematical Institute\\
	Endenicher Allee 60\\
	D--53115 Bonn \\
	Germany \&
	Instytut Matematyczny\\
	Uniwersytet Wroc{\l}awski\\
	Plac Grunwaldzki 2/4\\
	50--384 Wroc{\l}aw\\
	Poland}
 \email{mirek@math.uni-bonn.de}

\author{Christoph Thiele}
\address{
	C. Thiele\\
	Universit\"{a}t Bonn \\
	Mathematical Institute\\
	Endenicher Allee 60\\
	D--53115 Bonn \\
	Germany \& University of California, Los Angeles, Dep. of Mathematics, Los Angeles, CA 90095-1555, USA}
\email{thiele@math.uni-bonn.de}

\thanks{
M. Mirek was supported by the Hausdorff Center for Mathematics and
NCN grant DEC--2012/05/D/ST1/00053. 
C. Thiele was supported by the NSF grant DMS-1001535 and the Hausdorff
Center for Mathematics
}

\begin{abstract} We prove a multilinear local $T(b)$ theorem that
  differs from previously considered multilinear local $T(b)$ theorems
  in using exclusively general testing functions $b$ as opposed to a
  mix of general testing functions and indicator functions. The main
  new feature is a set of relations between the various testing
  functions $b$ that to our knowledge has not been observed in the
  literature and is necessitated by our approach. For simplicity we
  restrict attention to the perfect dyadic model.
\end{abstract}

\maketitle

\section{Introduction}

The theory of $T(1)$ and $T(b)$ theorems was started in the 1980's
papers \cite{DJ} and \cite{DJS} as a push to develop a general theory
of Calder\'on--Zygmund operators applicable for example in the
investigation of the Cauchy integral on Lipschitz curves. The first
{\it local} $T(b)$ theorem appears in Christ's paper \cite{C} with
applications to analytic capacity.  In recent years, the idea of {\it
  testing} that lies behind $T(1)$ and $T(b)$ theorems has become
influential in a wider array of topics related to singular integrals
such as for example sharp weighted estimates.

The topic of multilinear $T(1)$ theorems was discussed in the
companion papers \cite{GT},\cite{GT2}. More recently multilinear local
$T(b)$ theorems have been studied in \cite{GHO}. There a certain
square function is tested with general testing functions $b$, while
the dual operators to the operator in question are still tested with
characteristic functions $1$.  A {\it global} bilinear T(b) theorem
with testing functions $b$ throughout appears in \cite{Hart}.

In this paper we propose a multilinear local $T(b)$ theorem which only
tests with general testing functions $b$. This level of generality
appears to force a set of explicit constraints between the various
testing functions $b$, a phenomenon  which we did not find discussed in
the literature.  Clarification of the precise nature of these
constraints was a motivation for the present paper, as we encountered
the possibility of such constraints in the similar but more
complicated context of entangled operators in \cite{KT}, where as of
yet we have been unable to clarify the nature of an envisioned local
$T(b)$ theorem.

For simplicity we restrict attention to the perfect
Calder\'on--Zygmund setting discussed in \cite{AHMTT}. To gain
efficiency from symmetry we discuss multilinear forms which are dual
to multilinear operators. A dyadic cube in $\R^d$ is a cube whose
sides are dyadic intervals, that is intervals of the form $[2^m l,
2^m(l+1))$ with integers $m,l$.  A dyadic test function is a finite
linear combination of characteristic functions of dyadic cubes. An
$n$-linear form $\Lambda$ mapping $n$-tuples of dyadic test functions
to the set of real numbers is called a perfect Calder\'on--Zygmund
form if it satisfies the following three conditions:
\begin{enumerate}
\item[(i)] {\it Dyadic decay:} If each of the dyadic test function
  $f_1,\dots, f_n$ is supported on the same dyadic cube $P$, and if in
  addition two of these functions are supported on different dyadic
  children of $P$, where a dyadic child means a dyadic subcube of half
  of the sidelength, then
\begin{equation}\label{decay}
|\Lambda(f_1,\dots ,f_n)|\le \prod_{j=1}^n\|f_j\|_{p_j},
\end{equation}
where $(p_j)_{1\le j\le n}$ is any H\"older tuple of finite exponents,
that is $1<p_j< \infty$ and $\sum_{j=1}^n 1/p_j=1$.  This condition is
a dyadic version of standard pointwise decay estimates for
Calder\'on-Zygmund operators away from the diagonal.
\item[(ii)] {\it Perfect smoothness:} If one of the functions $f_j$ is
  supported on some dyadic cube $P$ and has mean zero, and if another
  one of the functions vanishes on that cube, then
\begin{equation}\label{smoothness}
\Lambda(f_1,\dots ,f_n)=0.
\end{equation}
This condition is a very strong dyadic version of standard decay estimates 
for derivatives of Calder\'on-Zygmund operators away from the diagonal.

\item[(iii)] {\it Qualitative truncation:} The integral kernel of the
  form $\Lambda$ is a dyadic test function in $\R^{dn}$. This
  condition is a dyadic version of standard truncation assumptions on
  Calder\'on-Zygmund operators, which are used to give sense to
  explicit integral formulas for Calder\'on-Zygmund operators but
  which do not usually enter a priori bounds for these operators in a
  quantitative way.

\end{enumerate}

Generally the idea behind a local $T(1)$ or $T(b)$ theorem is that a
H\"older estimate for an $n$-linear Calder\'on--Zygmund form can be
deduced from validity of the desired estimate for a very restricted
set of testing tuples of functions. We recall the perfect multilinear
local $T(1)$ theorem.

\begin{thm}[Perfect multilinear local $T(1)$ theorem]\label{perfectt1}
  Let $d\ge 1$ and $n\ge 2$ and let $\Lambda$ be a perfect $n$-linear
  Calder\'on--Zygmund form in $\R^d$. Let $1<p_j<\infty$ be a H\"older
  tupe of exponents.  Assume there is a constant $B\ge 1$ such that
$$
|\Lambda(f_1,\dots, f_{n})|\le B \prod_{j=1}^{n}\|f_j\|_{p_j}
$$
for all dyadic cubes $P$ and all tuples $(f_j)_{1\le j\le n}$ 
of functions such that all but one of the functions in this tuple
are the characteristic function of $P$ while the remaining function
is an arbitrary dyadic test function supported on $P$.

Then for some constant $C$ depending on $d$, $n$, $B$, and the tuple
$(p_j)_{1\le j\le n}$, we have
$$
|\Lambda(f_1,\dots, f_{n})|\le C \prod_{j=1}^{n}\|f_j\|_{p_j}
$$
for all $n$-tuples of dyadic test functions $(f_j)_{1\le j\le n}$.
\end{thm}

This theorem has been folklore in the field for some time, a continuous 
version of a
multilinear $T(1)$ theorem appears in \cite{GT}, and a proof
of Theorem \ref{perfectt1} can deduced 
from a similar theorem in \cite{KT}.

A $T(b)$ theorem is a variant of the $T(1)$ theorem, where the
characteristic functions of a cube $P$ are replaced by more general
functions $b$ which also have mean one on the cube $P$. In the present
paper we illustrate a multilinear local $T(b)$ theorem with a natural
set of interdependencies between the various functions $b$. To describe
these interdependencies we need some formal setup.

Let $I_n$ denote the set of integers $m$ with $1\le m \le n$.
A path in $I_n$ of length $k$ with $1\le k\le n$ is an injective 
mapping $\sigma: I_k\to I_n$.

We say that a collection $\Sigma$ of paths in $I_n$ is admissible if it
satisfies the following conditions:
\begin{enumerate}
\item For each $j\in I_n$ there is a path $\sigma\in \Sigma$ of length one with
  $\sigma(1)=j$.

\item For each path $\tilde{\sigma}\in \Sigma$ of any length $k<n$ 
  there is a path $\sigma\in \Sigma$ of length $k+1$ whose
  restriction to $I_k$ coincides with $\tilde{\sigma}$.

\item For each path $\sigma\in \Sigma$ of any length $k\ge 2$ there is a path $\tau\in \Sigma$ of the same length $k$ 
  which coincides with $\sigma$ on the set $I_{k-2}$ and satisfies
  $\sigma(k-1)=\tau(k)$ and $\tau(k-1)=\sigma(k)$.
\end{enumerate}

After stating Theorem \ref{perfecttb} below we give a fairly minimal example of an
admissible collection.  Let $\sigma$ be a path of length $k\le n$. We
say that an $n$-tuple $\Q$ of dyadic cubes is $\sigma$-nested if for
all $1\le i<j\le k$ we have
$$Q_{\sigma(i)}\supseteq Q_{\sigma(j)}$$
and whenever $s\in I_n$ is not in the range of $\sigma$ we have
$$Q_s= Q_{\sigma(k)};$$
in the case $k=n$ we additionally require that
$Q_{\sigma(n-1)}=Q_{\sigma(n)}$.
 
We now state the main new theorem in this paper.

\begin{thm}[Perfect multilinear local $T(b)$
  theorem] \label{perfecttb} Let $k\ge 1$ 
and  $n\ge \min(k,2)$ be integers.  Let $\Lambda$ be an $n$-linear form acting on
  $n$-tuples of dyadic test functions on the real line and being
  associated with a perfect Calder\'on--Zygmund kernel. Let
  $(p_j)_{1\le j\le n}$ be a H\"older-tuple of exponents.

  Assume we are given an admissible collection $\Sigma$ of paths and assume for
  each path $\sigma\in \Sigma$ of length $k$ and each
  $\sigma$-nested tuple $\Q$ of dyadic cubes, and each $j\in
  \sigma(I_{k-1})$ we are given a function $b_{\sigma, \Q, j}$, so
  that for some constant $B\ge 1$ the following properties are satisfied:
 \newline {\it
    Support condition:}
\begin{equation}\label{support}
{\rm supp}(b_{\sigma, \Q, j})\subseteq Q_{j}. 
\end{equation}
{\it Mean condition:}
\begin{equation}\label{mean} 
\int b_{\sigma,\Q,j}=|Q_j|.
\end{equation}
{\it Norm bound condition:} 
\begin{equation}\label{normbound}
\int |b_{\sigma,\Q,j}|^{p_j} \le B |Q_{j}|.
%\|\b_{\sigma,\Q,j}\|_{p_j}^{p_j} \le B |Q_{j}| .
\end{equation}
{\it Interdependence condition:} If $\sigma', \Q'$ and $1\le j<k$ are
such that we have for all 
$1\le l\le j$  $\sigma(l)=\sigma'(l)$
and $Q(l)=Q'(l)$, then we have
\begin{equation}\label{interdependence}
b_{\sigma , \Q, j}=b_{\sigma', \Q', j}.
\end{equation}
{\it Testing condition:} For all dyadic test functions
$g$ supported on $Q_{\sigma(k)}$ we have 
\begin{equation}\label{testing}
|\Lambda(f_1,\dots, f_{n})|\le B \prod_{j=1}^{n}\|f_j\|_{p_j},
\end{equation}
where $f_{\sigma(k)}=g$ and $f_{\sigma(l)}=\mathds{1}_{Q_{\sigma(k)}}
b_{\sigma,\Q,\sigma(l)}$ for $l<k$ and $f_{s}=\mathds{1}_{Q_{\sigma(k)}}$ for
$s$ which is not in the range of $\sigma$.

Then for some constant $C$ depending on $n$, $d$, the constant $B$,
and the H\"older tuple $(p_j)_{1\le j\le n}$, we have
\begin{equation}\label{multilinearcbound}
|\Lambda(f_1,\dots, f_{n})|\le C \prod_{j=1}^{n}\|f_j\|_{p_j}
\end{equation}
for any $n$-tuple $(f_j)_{1\le j\le n}$ of dyadic test functions.
\end{thm}

Note that the case $k=1$ of the local $T(b)$ theorem is the same as
the local $T(1)$ theorem. The strength of the theorem strictly
increases in $k$, as one can deduce the theorem for lower values of
$k$ by specializing some functions $b_{\sigma, \Q, j}$ to
characteristic functions. The case $k=n$ is the one of main
interest. We choose to introduce the parameter $k$ so as to induct on
it.

As an example of an admissible collection of paths, consider 
the collection of all paths that satisfy the following two properties:
\begin{enumerate}
\item[(i)]
The range of a path of length $k$ contains $I_{k-1}$.
\item[(ii)] If $j\le k$, then $I_j$ contains at least $j-1$
elements of the image of $I_j$ under the path.
\end{enumerate}
To see that this collection is admissible, first note that
the collection contains all paths of length one since conditions
(i) and (ii) are void for paths of length one. Hence the collection
satisfies (1).
Let $\tilde{\sigma}$ be a path of length $k<n$ in the collection.
Set $\sigma$ the path extending $\tilde{\sigma}$ by $\sigma(k+1)$
being the minimal element not in the range of $\tilde{\sigma}$.
Then $\sigma$ satisfies (i) since the range of $\tilde{\sigma}$
contains $I_{k-1}$ and if the range of $\tilde{\sigma}$  does not already
contain $I_k$ then $\sigma(k+1)=k$. To see that $\sigma$ satisfies
(ii), it suffices to check for $j=k+1$. But $I_{k+1}$ contains  at least 
$k$ elements of the range of $\sigma$ since $I_k$ contains at least 
$k-1$ elements and $\sigma(k+1)$ is at most $k+1$. Hence $\sigma$ 
satisfies (i) and (ii) and thus the collections satisfies (2).
Now let $\sigma$ be a path in the collection, and let $\tau$ be the
path described in (3). We need to show that $\tau$ is in the collection.
Property (i) is clear since the range of $\sigma$ equals that of $\tau$.
Property (ii) is only nontrivial for $j=k-1$. The property is clear
by monotonicity if $\tau(k-1)$ is in $I_{k-1}$. If $\tau(k-1)$
is not in $I_{k-1}$, then $\sigma(k)$ is not in $I_{k-1}$ and by
(i) we have that the range of $I_{k-1}$ under $\sigma$ is $I_{k-1}$.
This implies (ii) for $\tau$.

Another special case of our main theorem arises by choosing
$n$ appropriate functions $(b_j)_{1\le j\le n}$ and letting
$b_{\sigma,\Q,k}$ be suitably normalized restrictions of these
functions, that is with $j=\sigma(k)$,  
$$b_{\sigma,\Q,\sigma(k)}= b_{j} \mathds{1}_{Q_j}[b_{j} ]_{Q_j}^{-1},$$
where we have used the following notation for an average: 
$$[g]_P:=|P|^{-1}\int _P g.$$
We then obtain as straightforward corollary of the local theorem, following the
global/local reduction outlined in \cite{AHMTT}:

\begin{thm}[Perfect multilinear global $T(b)$
  theorem] \label{perfectglobaltb} Let $d\ge 1$ and $n\ge 2$ and let
  $\Lambda$ be an $n$-linear form acting on $n$-tuples of dyadic test
  functions on $\R^d$ associated with a perfect Calder\'on-Zygmund
  kernel. Let $(p_j)_{1\le j\le n}$ be a H\"older-tuple of exponents
  and assume we are given functions $(b_j)_{1\le j\le n}$ with the
  following properties:

\noindent{\it Pseudo-accretivity condition:}
For all dyadic cubes $Q$
\begin{equation}\label{globalmean} 
|[b_{j}]_Q|\ge 1.
\end{equation}
{\it Norm bound condition:} 
\begin{equation}\label{globalnormbound}
\|b_j\|_{\infty} \le B. 
\end{equation}
{\it Weak boundedness condition:} For all dyadic cubes $Q$ 
\begin{equation}\label{globalweaktesting}
|\Lambda(b_1\mathds{1}_Q,\dots, b_{n}\mathds{1}_Q)|\le B.
\end{equation}
{\it BMO condition:} For any $k$ and any dyadic test function $g$
\begin{equation}\label{globalstrongtesting}
|\Lambda(\dots, g, \dots )|\le B\|g\|_{H_1},
\end{equation}
where the $j$-th entry in the form is $b_j$ for $j\neq k$ and $g$ for
$j=k$, and $\|g\|_{H_1}$ is the norm of the dyadic Hardy space
(pre-dual of dyadic BMO).  Then for some constant $C$ depending on
$n$, $d$, the constant $B$, and the H\"older tuple $(p_j)_{1\le j\le
  n}$ we have
\begin{equation}\label{globalmultilinearcbound}
|\Lambda(f_1,\dots, f_{n})|\le C \prod_{j=1}^{n}\|f_j\|_{p_j}
\end{equation}
for any $n$-tuple $(f_j)_{1\le j\le n}$ of dyadic test functions.
\end{thm}
A bilinear continuous version of this theorem appears in \cite{Hart}.
Theorem \ref{perfecttb} arose from our efforts to adapt the techniques
of \cite{AHMTT} and subsequent papers to the multilinear setting,
setting up an induction on the number of functions $b$ that are not
characteristic functions.  In order to induct, we also refined the
technique in \cite{AHMTT} so as to use only multilinear estimates with
one fixed set of H\"older tuples. Our approach might give the reader 
new insights into the proof of the local $T(b)$ theorem for dyadic 
model operators in the linear case as well. 
We attempted to keep a maximal degree of symmetry in the argument
between dual versions of the same argument.

We outline briefly the aspect of precise exponents in the
norm bounds \eqref{normbound} on the testing functions.
In the earliest local $T(b)$ theorem, \cite{C}, Christ assumed that 
$b_Q^1, b_Q^2, Tb_Q^1,
T^*b_Q^2\in L^{\infty}$ uniformly with respect to $Q$. In \cite{NTV}
Nazarov, Treil and Volberg proved, in a non-doubling measure setup, that it
suffices to assume $b_Q^1, b_Q^2\in L^{\infty}$ and $Tb_Q^1, T^*b_Q^2\in
\mathrm{BMO}$ uniformly in $Q$. 
Auscher, Hofmann, Muscalu, Tao and Thiele
\cite{AHMTT} for dyadic model operators relaxed these
conditions assuming only $b_Q^1\in L^p$, $b_Q^2\in L^q$, $Tb_Q^1\in L^{q'}$
and $T^*b_Q^2\in L^{p'}$ for any $p, q\in (1, \infty)$ where the
different norms are appropriately scaled relative to $|Q|$, see also \cite{LV}.   

In 2008 Hofmann during his plenary lectures at the International
Conference on Harmonic Analysis and P.D.E. in El Escorial formulated
the question whether these testing conditions for the model dyadic
case also suffice for genuine singular integral operators. The
question was motivated by possible applications to layer potentials
and to free boundary theory.  Hofmann himself proved that it suffices
to assume $b_Q^1, b_Q^2\in L^{2+\varepsilon}$ and $Tb_Q^1, T^*b_Q^2\in
L^{2}$ for some $\varepsilon>0$. Auscher and Yang \cite{AY} eliminated
$\varepsilon>0$ from Hofmann's theorem by reducing the matters to the
dyadic case from \cite{AHMTT}. In fact they covered the sub-duality
case $1/p+1/q\le 1$. Auscher and Routin \cite{AR} covered the
super-duality case $1/p+1/q\ge1$ under some technical assumption
rather difficult to verify. Finally, Hyt\"onen and Nazarov \cite{HN}
provided the positive answer to Hofmann's question.

We comment one specific aspect of our proof: certain relatively
standard estimates near the end of the proof are accomplished using
the outer measure language from \cite{DT}.  We found this language
very useful here and hope the investment into understanding the novel
language will pay off in related questions of this kind, as it  has
been done in the present case.

A natural question which deserves for further investigation concerns
the extensions of our theorem to standard Calder\'on-Zygmund operators.

\section{Proof of the perfect multilinear local $T(b)$ theorem}

\subsection{General setup}
We prove Theorem \ref{perfecttb} by induction on $k$.
For $k=1$ the theorem specializes to Theorem \ref{perfectt1}
and this establishes the induction beginning.
Let $k\ge1$ and assume that the statement of 
Theorem  \ref{perfecttb} is true for this particular $k$.
We then have to prove the theorem with $k$ replaced by $k+1$.

Assume we are given $n\ge k+2$ and an $n$-linear perfect
Calder\'on--Zygmund form $\Lambda$ and an admissible collection
$\Sigma$ of paths in $I_n$.  For every admissible path $\sigma$ of
length $k+1$ and every $\sigma$-nested tuple $\Q$ and every $i\le k$
we are given $b_{\sigma,\Q,\sigma(i)}$ satisfying the assumptions in
Theorem \ref{perfecttb}.

For each admissible path $\tilde{\sigma}$ of length $k$ and each
$\tilde{\sigma}$-nested tuple $\Q$ and each $j<k$ 
we define
$$\tilde{b}_{\tilde{\sigma},\Q,\tilde{\sigma}(j)}:={b}_{\sigma,\Q,\sigma(j)},$$ 
where $\sigma$ is any admissible path of length $k+1$ extending the
path $\tilde{\sigma}$.  Note that such a path exists by the definition of
admissibility, that $\Q$ is also $\sigma$-nested, and that the
function on the right-hand side does not depend on the particular
choice of the extended path $\sigma$ by the interdependence assumption
\eqref{interdependence}. If $k+1=n$, then $\Q$ satisfies the
requirement $Q_{\sigma(n-1)}=Q_{\sigma(n)}$.

Then clearly this new set of testing functions satisfies the support
assumption \eqref{support}, the mean assumption \eqref{mean}, the norm
bound assumption \eqref{normbound}, and the interdependence assumption
\eqref{interdependence} of Theorem \ref{perfecttb} for $k$.  The main
part of the proof is to establish the testing condition
\eqref{testing} for this collection $\tilde{b}_{\tilde{\sigma},\Q,j}$
for some possibly new constant $B$ depending only on $n$, $d$, the
given constant $B$, and the tuple $(p_j)_{1\le j\le n}$.  Then
boundedness of $\Lambda$ follows by the induction hypothesis.

Let $A$ be the best constant in the inequality 

\begin{equation}\label{multilinearabound}
|\Lambda(f_1,\dots, f_{n})|\le A \prod_{j=1}^{n}\|f_j\|_{p_j}
\end{equation}
for any admissible path $\tilde{\sigma}$ of length $k$, any
$\tilde{\sigma}$-nested tuple $\Q$, and any dyadic test function
$f_{\tilde{\sigma}(k)}$, where $f_{\tilde{\sigma}(j)}=
\mathds{1}_{Q_{\tilde{\sigma}(k)}}\tilde{b}_{\tilde{\sigma},\Q,\tilde{\sigma}(j)}$
whenever $j< k$, and where $f_{s}=\mathds{1}_{Q_{\tilde{\sigma}(k)}}$
for any $s$ which is not in the range of $\tilde\sigma$.

By the truncation assumption on the form $\Lambda$, the constant $A$
is finite.  We will show that $A$ can be estimated by a constant
depending only on $n$, $d$, $B$, and the tuple $(p_j)_{1\le j\le n}$,
which will establish the testing assumption \eqref{testing} for the
collection $\tilde{b}_{\tilde{\sigma},\Q,j}$. Then the induction
hypothesis will do the job and Theorem \ref{perfecttb} is established
for $k+1$.

Let $\tilde{\sigma}$ be an admissible path of length $k$ and $\Q$ a
$\tilde{\sigma}$-nested tuple and $f_{\tilde{\sigma}(k)}$ a dyadic
test function such that equality in \eqref{multilinearabound} is
attained for this data. Since $\Lambda$ has finite rank, such extremal
point exists. Indeed, the extremal function $f_{\tilde{\sigma}(k)}$
can be chosen to be a dyadic test function.

Let $\sigma$ be an admissible extension of $\tilde{\sigma}$ of length $k+1$.
So far the entire setup is symmetric under permutation of the 
numbers $1,\dots,n$, so to simplify notation by symmetry we may assume $\sigma$ 
is the path
$$\sigma(j)=j$$
for $j< k$ and 
$$\sigma(k)=k+1,\ \sigma(k+1)=k.$$
We shall also need the path $\tau$ of length $k+1$ which interchanges
the last two steps of $\sigma$, that is the identity embedding
$\tau(j)=j$ for $j\le k+1$.
Note that $b_{\sigma,\Q,j}$ coincides with $b_{\tau,\Q,j}$ for
$j< k$ by the interdependence assumption \eqref{interdependence}.
Note also that $Q_k=Q_{k+1}$. We set $Q:=Q_k=Q_{k+1}$.

For a dyadic cube  $P$ define with
$f_{\sigma(j)}= \mathds{1}_Q \tilde{b}_{\tilde{\sigma},\Q,\tilde{\sigma}(j)}$ whenever $j< k$,
and $f_{s}=\mathds{1}_{Q}$ for any $s$ which is not in the range of $\sigma$:
$$\Lambda_P(\varrho ,\tilde{\varrho}):=\Lambda(f_1 \mathds{1}_P, \dots ,
f_{k-1}\mathds{1}_P, 
\varrho \mathds{1}_P, \tilde{\varrho}\mathds{1}_P,f_{k+2}\mathds{1}_P , \dots,f_n \mathds{1}_P), $$
that is the $k$-th and $k+1$-st entry are $\varrho \mathds{1}_P$ and
$\tilde{\varrho}\mathds{1}_P$ respectively, while the $j$-th entry with
$j\neq k,k+1$ is $f_j \mathds{1}_P$. Then we have
\begin{equation}\label{multilinearaequality}
|\Lambda_{Q}(f_k, f_{k+1})|=A \prod_{j=1}^{n}\|f_j\|_{p_j}
\end{equation}
with $f_k=\mathds{1}_{Q}$ and with $f_{k+1}$ the chosen extremizing function.

\subsection{The first stopping time}

We consider the setup of the previous section, in particular the paths $\sigma,\tau$
and the tuple $\Q$ have these specific meanings, as well as the chosen functions $f_j$.
We introduce the abbreviations
$$g:=f_{k+1}, \ h:=f_k,$$
$$q:=p_{k+1}, \ r:=p_k. $$
We also abbreviate the particular testing function  
$b_{\sigma,\Q,k+1}$ by $u$.
We continue to write
$[\varrho]_P$ for the average of a function $\varrho$ over a
cube $P$, and we write $\varrho_P$ for the truncated function
$\varrho \mathds{1}_P$. We define 
$$F:= \prod_{j\neq k,k+1}\|f_j\|_{p_j}=\prod_{j\neq k,k+1}\|f_j\mathds{1}_Q\|_{p_j}.$$
Finally, we choose  $\varepsilon>0$ small enough so that 
$$\varepsilon= (1/8)\min_{1\le j\le n}\left( 7/8 \right)^{{p_j}/{(p_j-1})} B^{-{1}/{(p_j-1)}}. $$

We define a stopping time inside $Q$, that is a collection of pairwise
disjoint cubes contained in $Q$ which have good properties relative to the functions
$f_j$, the form $\Lambda$, and the testing functions.

Let $\P_1$ be the collection of maximal dyadic cubes 
$P$ contained in $Q$ for which there exists a $1\le j\le n$ with 
\begin{equation}\label{stopupperthresholdnorm}
|P|^{-1}\|f_j \mathds{1}_P \|_{p_j}^{p_j}\ge  n \varepsilon^{-1} 
|Q|^{-1}\|f_j\|_{p_j}^{p_j}.
\end{equation}
Let $\P_2$ be the collection of maximal dyadic cubes 
$P$ contained in $Q$ which satisfy 
\begin{equation}\label{stopupperthresholdnormb}
|P|^{-1}\|u_P \|_{q}^{q}\ge  \varepsilon^{-1} 
|Q|^{-1}\|u \|_{q}^{q}.
\end{equation}
Let $\P_3$ be the collection of maximal dyadic cubes
$P$ contained in $Q$ for which there exists a
nonzero function $\varrho $ supported on $P$ with mean zero such that
\begin{equation}\label{stopupperdualthreshold}
|\Lambda_P(\varrho , u)|\ge B F \varepsilon^{-1}(|P|/|Q|)^{1-1/r} \|\varrho\|_{r} \|u\|_{q}.
\end{equation}
Let $\P_4$ be the collection of maximal dyadic cubes 
$P$ contained in $Q$ which satisfy 
\begin{equation}\label{stoplowerthresholdmean}
|[u]_P| \le  1/{8}.
\end{equation}
Let $\P_5$ be the collection of maximal dyadic cubes 
$P$ contained in $Q$ which are contained in
at least $2^d \varepsilon^{-1}$ many dyadic cubes which are
parents of cubes in $\P_4$.
Let $\P$ be the collection of maximal dyadic cubes in 
$\P_1\cup \P_2\cup \P_3\cup \P_4\cup \P_5$.

We claim that 
\begin{equation}\label{firststopping}
\sum_{P\in \P} |P|\le (1-\varepsilon)|Q|.
\end{equation}
To verify the claim, we discuss the sets $\P_1$ through $\P_5$ separately.
The collection $\P_1$ consists of pairwise disjoint cubes and satisfies
$$\sum_{P\in \P_1} |P|\le \varepsilon n^{-1} |Q| \sum_{P\in \P_1}
\sum_{j=1}^n \|f_j \mathds{1}_P \|_{p_j}^{p_j} \|f_j \|_{p_j}^{-p_j}
\le \varepsilon |Q|.$$ The collection $\P_2$ is estimated similarly.
To estimate $\P_3$, consider for each $P\in \P_3$ a function $\varrho _P$
supported on $P$ with mean zero satisfying $\|\varrho_P\|_{r}^{r}=|P|$
and inequality \eqref{stopupperdualthreshold} without the absolute
value on the left-hand side.  Then we have with the
testing assumption \eqref{testing} for $\sigma$, $\Q$, $k$:
$$\sum_{P\in \P_3}  B F \varepsilon^{-1}|P||Q|^{-1+1/r} \|u\|_{q} 
\le \sum_{P\in \P_3}\Lambda_P(\varrho_P,u)\le \Big|\Lambda_{Q}\Big(\sum_{P\in \P_3} \varrho_P,u\Big)\Big|$$
$$\le B F \Big\|\sum_{P\in \P_3} \varrho_P\Big\|_{r}
\|u\|_{q} 
\le BF |Q|^{1/r}
\|u\|_{q}.$$
Hence
$$\sum_{P\in \P_3}  |P|\le \varepsilon |Q|.$$
To estimate the collection $\P_4$, set $E=Q\setminus \bigcup_{P\in\P_4} P$. Then
$$|E|^{1-1/q}\|u\|_{q}\ge \Big|\int_E 
u\Big| =|Q| -\sum_{P\in \P_4}|P||[u]_P|\ge |Q|-\sum_{P\in \P_4} (1/8) |P|\ge (7/8) |Q| .$$
This implies
$$|E|^{1-1/q}B^{1/q}\ge (7/8)  |Q|^{1-1/q},$$
$$|E| \ge (7/8)^{q/(q-1)} B^{-1/(q-1)} |Q|,$$
which in turn implies 
$$\sum_{P\in \P_4}|P|\le (1-8 \varepsilon)|Q|.$$
Finally, we have the estimate
$$\sum_{P\in \P_5}|P|\le (2^{-d}\varepsilon) \sum_{P\in \P_4} 2^d |P|\le \varepsilon|Q|.$$
Adding the contributions from $\P_1$ through $\P_5$  proves the claim.

We call the cubes in $\P_1\cup \ldots \cup \P_5$ the stopping cubes.
We note that if $P$ is not contained in any child or grandchild of a stopping cube,
then we have the following upper bounds with a constant $C$ depending only
on $d$, $n$, $B$, 
\begin{equation}\label{upperthresholdnorm}
|P|^{-1}\|f_j \mathds{1}_P \|_{p_j}^{p_j}\le  C |Q|^{-1}\|f_j \|_{p_j}^{p_j},
\end{equation}
\begin{equation}\label{upperthresholdnormb}
|P|^{-1}\|u_P \|_{q}^{q}\le  C 
|Q|^{-1}\|u \|_{q}^{q}\le C,
\end{equation}
and 
\begin{equation}\label{upperdualthreshold}
|\Lambda_P(\varrho , u)|\le C (|P|/|Q|)^{1-1/r} \|\varrho\|_{r}\|u\|_{q}\prod_{j\neq k,k+1}\|f_j\|_{p_j}
\end{equation}
for any function $\varrho$ supported on $P$ and with mean
zero. For the cubes $P$ not contained in any stopping cube this is
clear by the construction. For the stopping cubes themselves or their
children this follows by observing the estimate for the parent or
grandparent of the cube and deducing the estimate with a modified
constant for the cube itself. Such passage to the stopping cubes
applies only for the upper bounds listed above, the threshold
\eqref{stoplowerthresholdmean} leads to the lower bound
\begin{equation}\label{lowerthresholdmean}
|[u]_P| \ge 1/{8}
\end{equation}
only for all cubes not contained in a stopping cube of type $\P_4$.  It does not yield analoguous 
lower bounds for the stopping cubes in $\P_4$ themselves. This is the reason for introducing the collection $\P_5$
and the special arguments concerning $\P_4$ below.

\subsection{Pruning the function $g$}

In this section we replace the function $g$ with a modified function
$\mathfrak g$ which is adapted to the first stopping time.

Let $\overline{\P}_4$ be the collection of parents of dyadic cubes in 
$\P\cap \P_4$
and let ${\P'_4}$ be the collection of dyadic cubes which are siblings of
cubes in $\P\cap \P_4$ but not themselves cubes in $\P\cap \P_4$.

Define 
\begin{equation}\label{ftildeexpansion}
\mathfrak g=g-[g]_{Q} u
-\sum_{P\in \P\setminus \P_4} g_P 
+\sum_{P\in \P\setminus \P_4} \frac{[g]_P}{ [u]_P}u_P
-\sum_{P\in \P\cap \P_4} g_P +\sum_{P\in \overline{\P}_4} \frac{[g]_P}{ [u]_P}u_P
-\sum_{P\in {\P'_4}} \frac{[g]_P}{ [u]_P}u_P
\end{equation}
and note that $\mathfrak g$ is still supported on $Q$ and $[\mathfrak g]_Q=0$ and
$\|\mathfrak g\|_q^q\le C\|g\|_q^q$.

We claim that the desired bound for $A$ in \eqref{multilinearaequality} follows from
\begin{equation}
  \label{eq:1}
  |\Lambda_{Q}(h,g-\mathfrak g)|\le \big(A(1-\varepsilon)^{1/p_k}+C\big)F \|g\|_q\|h\|_r
\end{equation}
and
\begin{equation}\label{multilineartildeabound}
  |\Lambda_{Q}(h,\mathfrak g)|\le 
  C F \|g\|_q\|h\|_r.
\end{equation} 
Here and in the sequel $C$ denotes a constant which depends only on $B$,
$n$, $d$, and the tuple $(p_j)_{1\le j\le n}$, but may vary from line
to line.  Indeed, to verify the claim, it suffices now to expand in identity
\eqref{multilinearaequality} the function $g$ into $\mathfrak g$ plus
correction term and make use of \eqref{eq:1} and
\eqref{multilineartildeabound}.  
Then dividing both sides by the product of norms we obtain
$$A\le A(1-\varepsilon)^{1/p_k}+C.$$
Solving $A$ from this inequality the desired bound for $A$ is
established and the matters are reduced to proving \eqref{eq:1} and
\eqref{multilineartildeabound}.

We begin with the bound in \eqref{eq:1} and estimate separately the
contributions of the various terms of the difference $g-\mathfrak
g$. For this purpose for each $P\in \P$ we add and
subtract to $g-\mathfrak g$ a new term involving the function
$b_{\sigma, \Q_P, k+1}$ associated with the permutation $\sigma$ and
the chain $\Q_P$ given by
$$Q_1\supseteq Q_2\ldots \supseteq Q_{k-1}\supseteq P\supseteq \ldots \supseteq P, $$
that is the chain $\Q_P$ coincides with $\Q$ up to entry $k-1$ and then stabilizes to
$P$.  Therefore, we obtain
\begin{align}
  \label{eq:10}
  g-\mathfrak g=[g]_{Q} u
+\sum_{P\in \P} (g_P-[g]_Pb_{\sigma, \Q_P, k+1})
\end{align}
\begin{align}
  \label{eq:13}
+
\sum_{P\in \P\setminus \P_4}\Big([g]_Pb_{\sigma, \Q_P, k+1}- \frac{[g]_P}{ [u]_P}u_P\Big)
\end{align}
\begin{align}
  \label{eq:14}
+
\sum_{P\in \P\cap \P_4} [g]_Pb_{\sigma, \Q_P, k+1} -\sum_{P\in \overline{\P}_4} \frac{[g]_P}{ [u]_P}u_P
+\sum_{P\in {\P'_4}} \frac{[g]_P}{ [u]_P}u_P.
\end{align}

We have by multilinearity and the testing assumption 
\eqref{testing}
$$|\Lambda_{Q}(h,[g]_{Q}u)|\le 
C F [g]_{Q} \|u\|_{q} \|h\|_r \le C F \|g\|_{q} \|h\|_r .$$ In the
second inequality we have estimated the mean of $g$ by H\"older's
inequality and the norm of $u$ by the norm bound assumption
\eqref{normbound}.  This establishes the desired estimate for the
first term of \eqref{eq:10} in the expansion of $g-\mathfrak g$.

Next we consider the sum in \eqref{eq:10} involving the stopping cubes 
from $\P$.   We  calculate
with multilinearity and the smoothness condition \eqref{smoothness}
$$\Lambda_{Q}\Big(h,\sum_{P\in \P} (g_P -[g]_P b_{\sigma,\Q_P,k+1})\Big)
=\sum_{P\in \P} \Lambda_P(h,g  -[g]_P b_{\sigma,\Q_P,k+1})$$
\begin{equation}\label{aandcnot4}
= \sum_{P\in \P} \Lambda_P(h, g)
- \sum_{P\in \P} \Lambda_P(h, [g]_P
b_{\sigma,\Q_P,k+1}).
\end{equation}
For the first summand in \eqref{aandcnot4} we use  estimate
\eqref{multilinearabound} with $A$
for the data $\tau$, $\Q_P$  and obtain by H\"older's inequality
$$ \sum_{P\in \P} |\Lambda_P(h, g)|\le A \sum_{P\in \P} \prod_{j=1}^{n} \|f_j\mathds{1}_P\|_{p_j}
 \le A 
\prod_{j=1}^{n} \Big(\sum_{P\in \P} \|f_j\mathds{1}_P\|_{p_j}^{p_j}\Big)^{1/p_j}$$
$$ \le A F 
\Big(\sum_{P\in \P}|P| \Big)^{1/r} \|g\|_q  \le A F (1-\varepsilon)^{1/r}
\|g\|_q\|h\|_r.$$
Here we have used that $h=\mathds{1}_Q$ and that $\sum_{P\in\P} |P|\le (1-\varepsilon)|Q|$.
The second term in \eqref{aandcnot4} we estimate by the testing assumption \eqref{testing} 
with the data
$\sigma$, $\Q_P$ and obtain
$$
\sum_{P\in \P} |\Lambda_P(h, [g]_P
b_{\sigma,\Q_P,k+1})|\le
C 
\sum_{P\in \P} |[g]_P|  \|b_{\sigma,\Q_P,k+1}\|_{q}\|h_P\|_r \prod_{j\neq k,k+1}\|f_j\mathds{1}_P\|_{p_j}$$
$$\le C 
\sum_{P\in \P}  \|g_P\|_q\|h_P\|_r\prod_{j\neq k,k+1} 
\|f_j\mathds{1}_P\|_{p_j}
\le C F \|g\|_q\|h\|_r . $$
We have used the upper bounds (\ref{upperthresholdnorm}) including the cases $f_{k+1}=g$ and $f_k=h$.

We now consider the sum  in \eqref{eq:13}. We use a vanishing mean again to write
$$\Lambda_{Q}\Big(h, 
\sum_{P\in \P\setminus \P_4}[g]_P b_{\sigma,\Q_P,k+1}
- \frac{[g]_P}{ [u]_P}u_P
\Big)$$
\begin{equation}\label{termsnot4}
=
\sum_{P\in {\P}\setminus \P_4}
\Lambda_P(h , 
{[g]_P}b_{{\sigma}, {\Q_P}, k+1})
- \sum_{P\in {\P}\setminus \P_4}
\Lambda_P\Big(h , 
\frac{[g]_P}{ [u]_P}u\Big)
\end{equation}

The first term in \eqref{termsnot4} is estimated similarly as before.
To estimate the second term in \eqref{termsnot4} we add and subtract a
term, involving the function $b_{\tau, \Q_P,k}$ associated with the
path $\tau$ and the chain $\Q_P$ as above, so that we obtain for that
term
\begin{equation}\label{asiotapnot4}
\sum_{P\in {\P}\setminus \P_4}
\Lambda_P\Big([h]_P b_{\tau, \Q_P, k}, 
\frac{[g]_P}{ [u]_P}u\Big)
+\sum_{P\in {\P}\setminus \P_4}
\Lambda_P\Big(h  - [h]_P b_{\tau, \Q_P, k}, 
\frac{[g]_P }{[u]_P}u\Big) .
\end{equation}

The first term in \eqref{asiotapnot4} is estimated by the testing
assumption \eqref{testing} with the data $\tau, \Q_P$ by
$$ C 
\sum_{P\in {\P}\setminus \P_4}
\bigg|\frac{[h]_P [g]_P}{ [u]_P} \bigg| \|b_{\tau, \Q_P,
  k}\|_{r}\|u_P\|_{q}\prod_{j\neq k,k+1}\|f_j\mathds{1}_P\|_{p_j}\le C
F \|g\|_q\|h\|_r . $$ Here we have used similarly as above the upper
bounds \eqref{upperthresholdnorm} and \eqref{upperthresholdnormb} and
\eqref{lowerthresholdmean}, and that the cubes in ${\P}\setminus \P_4$
are not contained in any of the cubes $\P_4$ and thus satisfy
\eqref{lowerthresholdmean}.  We also used \eqref{normbound} for
$b_{\tau,\Q_P,k}$ and \eqref{firststopping}.

The second term in \eqref{asiotapnot4} we estimate with
\eqref{upperdualthreshold} by
$$ CF \sum_{P\in {\P}\setminus \P_4}(|P|/|Q|)^{1-1/r}  
\bigg|\frac{[g]_P}{ [u]_P} \bigg|\left\|h_P-[h]_Pb_{\tau,\Q_P,k}\right\|_{r}
\|u\|_{q} \le 
C F \|g\|_q\|h\|_r. $$

The terms of the expansion \eqref{eq:14} of $g-\mathfrak g$ involving
$\overline{\P}_4$ are similar but slightly more involved since the
lower bound on the average of $b_{\sigma,\Q,k+1}$ is not available and
one has to therefore work with parent and sibling cubes.  Let
$\overline{P}$ denote the parent cube of a dyadic cube $P$.  
Then we rewrite  \eqref{eq:14} as $\sum_{P\in
  \overline{\P}_4} \xi^P$, where for each $P\in \overline{\P}_4$ the
function $\xi^P$ is defined as
\begin{equation}\label{bufferterms}
- \frac{[g]_P }{ [u]_P}u_P
+\sum_{P'\in {\P'_4}:\overline{P'}=P} \frac{[g]_{P'}}{ [u]_{P'}}u_{P'}
+\sum_{P'\in\P\cap\P_4: \overline{P'}= P}[g]_{P'} b_{\sigma,\Q_{P'},k+1}.
\end{equation}
Exactly $P$ and the children of $P$ contribute to $\xi^P$. Note that
at least one child of $P$ is in $\P\cap \P_4$. The mean of $\xi^P$ is
zero. Hence we can write with the smoothness condition
\eqref{smoothness}
$$\sum_{P\in \overline{\P}_4}\Lambda_{Q}(h, \xi^P)=
\sum_{P\in \overline{\P}_4}\Lambda_P(h, \xi^P).$$
Expanding $\xi^P$ again into three terms as in \eqref{bufferterms} and considering the terms 
separately, we obtain in analogy to \eqref{termsnot4}
\begin{equation}\label{terms}
\sum_{P\in \P\cap \P_4}
\Lambda_{\overline{P}}(h , 
{[g]_P}b_{\sigma, {\Q_P}, k+1})
- \sum_{P\in \overline{\P}_4}
\Lambda_P\Big(h , 
\frac{[g]_P}{ [u]_P}u\Big)
+ \sum_{P\in {\P'_4}}
\Lambda_{\overline{P}}\Big(h , 
\frac{[g]_P}{ [u]_P}u\Big)
\end{equation}
To estimate the first term in \eqref{terms}, we write for each $1\le j\le n$
$$f_j \mathds{1}_{\overline{P}}=f_j \mathds{1}_{P}+\sum_{P'\neq P, \overline{P'}=\overline{P}} f_j \mathds{1}_{P'}$$
and expand the multilinear form correspondingly.
Any term in the expansion which has a $f_j\mathds{1}_{P'}$ for  some $j$ can be estimated
by the decay condition \eqref{decay} so that we obtain for the penultimate
display the upper bound
$$\Big|\sum_{P\in \P\cap \P_4}\Lambda_{P}(h, [g]_P b_{\sigma,\Q_P,k+1})\Big|
+C \sum_{P\in \P_4} \prod_{j=1}^n \|f_j\mathds{1}_{\overline{P}}\|_{p_j}
\le C F\|g\|_q\|h\|_r.$$
Here we have estimated the first term as for the cubes $\P\setminus \P_4$ and we have
applied stopping conditions as before.

The second and third terms in \eqref{terms} are estimated similarly to
the case of cubes in $\P\setminus \P_4$. The cubes in
$\overline{\P}_4$ are not pairwise disjoint, but they have bounded
overlap since they are not contained in any cube of $\P_5$ by
construction.  Similarly the cubes in ${\P'_4}$ have bounded
overlap. This completes the proof of \eqref{eq:1}. The proof of
Theorem \ref{perfecttb} will be completed if we establish
\eqref{multilineartildeabound}.

\subsection{The second stopping time and pruning the function $h$}

Now let $A'$ be the best constant so that for all dyadic cubes
$R\subseteq Q$ we have the estimate
\begin{equation}\label{multilinearrbound}
|\Lambda_{R}(h,\mathfrak g-[\mathfrak g]_R[u]_R^{-1}u)|
\le 
A'F
|R||Q|^{-1}  \|g\|_q\|h\|_r.
\end{equation}
The constant $A'$ is again finite since $\Lambda$ satisfies the
truncation assumption.  We will show that $A'$ can be estimated from
above by a constant $C$ depending only on $n$, $d$, $B$, and
$(p_j)_{1\le j\le n}$. This will establish
\eqref{multilineartildeabound} by setting $R=Q$ since $[\mathfrak g]_{Q}=0$.

Fix a dyadic cube $R$ such that equality in \eqref{multilinearrbound}
is attained. Again such a cube $R$ exists since $\Lambda$ satisfies the
truncation assumption.  We may assume that $R$ is not contained in any
stopping cube of the first stopping time since for such cubes
$(\mathfrak g-[\mathfrak g]_R[u]_R^{-1}u) \mathds{1}_R=0$.  We set
$\hat{\mathfrak g}:=(\mathfrak g-[\mathfrak g]_R[u]_R^{-1}u) \mathds{1}_R$.

Consider the functions $b_{\tau, \Q_R,j}$ where $\Q_R$ is the chain
$$Q_1\supseteq Q_2\supseteq \ldots \supseteq Q_{k-1}\supseteq R\supseteq \ldots \supseteq R.$$
For the simplicity we shall write $v$ for $b_{\tau,\Q_R,k}$.

We invoke a second stopping time.
Let $\S_1$ be the collection of maximal dyadic cubes 
$S$ contained in $R$ for which there exists 
$1\le j\le {n}$ with
\begin{equation}\label{stopuppertnorm2}
|S|^{-1}\|f_j \mathds{1}_S \|_{p_j}^{p_j}\ge  n \varepsilon^{-1} 
|R|^{-1}\|f_j\mathds{1}_R \|_{p_j}^{p_j}.
\end{equation}
or
\begin{equation}\label{stopuppertnormb2tilde}
|S|^{-1}\| \mathfrak g_S \|_{q}^{q}\ge  \varepsilon^{-1} 
|R|^{-1}\|\mathfrak g_R\|_{q}^{q}.
\end{equation}
Let $\S_2$ be the collection 
of maximal dyadic cubes 
$S$ contained in $R$ which satisfy 
\begin{equation}\label{stopuppertnormb2}
|S|^{-1}\|v_S \|_{r}^{r}\ge 2 \varepsilon^{-1} 
|R|^{-1}\|v\|_{r}^{r}
\end{equation}
or 
\begin{equation}\label{stopuppertnormb3}
|S|^{-1}\|u_S \|_{q}^{q}\ge 2 \varepsilon^{-1} 
|R|^{-1}\|u_R\|_{q}^{q}
\end{equation}
Let $\S_3$ be the collection of maximal dyadic cubes
$S$ contained in $R$ for which there exists a
nonzero function $\varrho$ supported on $S$ with mean zero such that
\begin{equation}\label{stopupperdualthreshold2}
|\Lambda_S(v,\varrho)|\ge B \varepsilon^{-1}(|S|/|R|)^{1-1/q} \|\varrho\|_{q} \|v\|_{r}\prod_{j\neq k,k+1}\|f_j\mathds{1}_R\|_{p_j},
\end{equation}
or there exists a
nonzero function $\varrho $ supported on $S$ with mean zero such that
\begin{equation}
\label{eq:15}
|\Lambda_S(\varrho , u)|\ge B  \varepsilon^{-1}(|S|/|R|)^{1-1/r} \|\varrho\|_{r} \|u_R\|_{q}\prod_{j\neq k,k+1}\|f_j\mathds{1}_R\|_{p_j}.
\end{equation}
Let $\S_4$ be the collection of maximal dyadic cubes $S$ contained
in $R$ which satisfy 
\begin{equation}\label{stoplowertmean2}
|[v]_S| \le 1/{8}.
\end{equation}
Let $\S_5$ be the collection of maximal dyadic cubes 
$S$ contained in $R$ which are contained in
at least $2^d \varepsilon^{-1}$ many dyadic cubes which are
parents of cubes in $\S_4$.
Let $\S$ be the collection of maximal dyadic cubes in 
$\S_1\cup \S_2\cup \S_3\cup \S_4\cup \S_5$.

Then we have similarly as for the first stopping time
\begin{equation}\label{firststopping2}
\sum_{S\in \S} |S|\le (1-\varepsilon)|R|.
\end{equation} 
Also we obtain the following upper in analogy to the first stopping time. If
$S$ is not contained in any child of a stopping cube, then we have the
following upper bounds for $1\le j<n$ with a constant $C$ depending
only on $d$, $n$, $B$,
\begin{equation}\label{uppertnorm2}
|S|^{-1}\|f_j \mathds{1}_S\|_{p_j}^{p_j}\le  C |R|^{-1}\|f_j\mathds{1}_R \|_{p_j}^{p_j}\le 
C |Q|^{-1}\|f_j \|_{p_j}^{p_j}.
\end{equation}
In the latter inequality we have used \eqref{upperthresholdnorm}
and the fact that $R$ is not contained in any stopping cube of the first stopping time. We also have 
\begin{align}
  \label{eq:9}
|S|^{-1}\|\mathfrak g_S\|_{q}^{q}\le  C |R|^{-1}\|\mathfrak g_R \|_{q}^{q}\le 
C |Q|^{-1}\|g \|_{q}^{q}.  
\end{align}
Here the last inequality follows by estimating the various terms in
the expansion of $\mathfrak g$. Indeed, we have
$$\|g_R\|_{q}^{q}\le |R||Q|^{-1}
\|g\|_{q}^{q} $$
since $R$ is not contained in any stopping cube $P$ of the first stopping time.
We also have for the same reason
$$\|[g]_{Q} u_R\|_{q}^{q}\le C|R||Q|^{-1} \|g\|_{q}^{q}. $$
Next we have
$$\Big\|\sum_{P\in \P\setminus\P_4} g_P\mathds{1}_R \Big\|_{q}^{q}\le |R||Q|^{-1}\|g\|_{q}^q$$
by disjointness of the cubes $P\in \P$ and by the upper bounds from
\eqref{upperthresholdnorm}. Similarly
$$\Big\|\sum_{P\in \P\setminus\P_4} \frac{[g]_P}{[u]_P}u_P  \mathds{1}_R\Big \|_{q}^{q}\le |R||Q|^{-1}\|g\|_{q}^q .$$
Similarly we estimate the terms corresponding with the stopping cubes $P\in
\P_4\cap \P$. For  the cubes from $\overline{\P}_4$ and
${\P'_4}$ we use their bounded overlapping. We further obtain the upper bounds
\begin{equation}\label{uppertnormb2}
|S|^{-1}\|v_S\|_{r}^{r}\le  C 
|R|^{-1}\|v\|_{r}^{r}
\end{equation}
and
\begin{equation}\label{uppertnormb3}
|S|^{-1}\|u_S\|_{q}^{q}\le  C 
|R|^{-1}\|u_R\|_{q}^{q}\le C|Q|^{-1}\|u\|_q
\end{equation}
and 
\begin{equation}\label{upperdualt2}
|\Lambda_S(v, \varrho)|\le C F (|S|/|Q|)^{1-1/q}(|R|/|Q|)^{-1/r} \|\varrho\|_{q}\|v\|_{r}
\end{equation}
and
\begin{align}
  \label{eq:16}
  |\Lambda_S(\varrho , u)|\le CF(|S|/|Q|)^{1-1/r} \|\varrho\|_{r} \|u\|_{q}
\end{align}
for any function $\varrho$ supported on $S$ and with mean
zero.

Let $\overline{\S}_4$ be the collection of parents of dyadic cubes in $\S\cap \S_4$
and let ${\S'_4}$ be the collection of dyadic cubes which are siblings of
cubes in $\S\cap \S_4$ but not themselves cubes in $\S\cap \S_4$. Define 
\begin{equation}\label{htildeexpansion}
\mathfrak h=h_R-[h]_{R} v
-\sum_{S\in \S\setminus \S_4} h_S 
+\sum_{S\in \S\setminus \S_4} \frac{[h]_S}{ [v]_S}v_S
-\sum_{S\in \S\cap \S_4} h_S +\sum_{S\in \overline{\S}_4} \frac{[h]_S}{ [v]_S}v_S
-\sum_{S\in {\S'_4}} \frac{[h]_S}{ [v]_S}v_S
\end{equation}
and note that $\mathfrak h$ is supported on $R$ and $[\mathfrak h]_R=0$.
As in the first stopping time, the desired bound for $A'$ follows from
\begin{equation}
  \label{eq:2}
  |\Lambda_{R}(h-\mathfrak h,\hat{\mathfrak g})|\le
  \big(A'(1-\varepsilon)+C\big)F |R||Q|^{-1} \|g\|_q\|h\|_r
\end{equation}
and
\begin{equation}\label{tildehatbound}
|\Lambda_{R}(\mathfrak h,\hat{\mathfrak g})|\le C F |R||Q|^{-1}
\|g\|_q\|h\|_r.
\end{equation}
where $C$ may depend on $B$, $n$, $d$, and the tuple $(p_j)_{1\le j\le
  n}$.  Arguing similarly as in the proof of \eqref{eq:1} one obtains
\eqref{eq:2}. In what follows we repeat this argument
 with the necessary minor changes.

 To obtain the bound in \eqref{eq:2} we estimate separately the
 contributions of the various terms of the difference $h_R-\mathfrak
 h$. For doing so, for each $S\in \S$ we add and
subtract to $h_R-\mathfrak
 h$ a new term involving the function $b_{\tau, \Q_S, k}$
associated with the permutation $\tau$ and the chain $\Q_S$ given by
$$Q_1\supseteq Q_2\ldots \supseteq Q_{k-1}\supseteq S\supseteq \ldots \supseteq S, $$
that is the chain $\Q_S$ coincides with $\Q_R$ up to entry $k-1$ and then stabilizes to
$S$. Then we obtain
\begin{align}
  \label{eq:23}
  h_R-\mathfrak h=[h]_{R} v +
\sum_{S\in \S} (h_S-[h]_Sb_{\tau, \Q_S, k})
\end{align}
\begin{align}
  \label{eq:24}
  +\sum_{S\in \S\setminus \S_4} \Big([h]_Sb_{\tau, \Q_S, k}-\frac{[h]_S}{ [v]_S}v_S\Big)
\end{align}
\begin{align}
  \label{eq:25}
 + \sum_{S\in \S\cap \S_4} [h]_Sb_{\tau, \Q_S, k} -\sum_{S\in \overline{\S}_4} \frac{[h]_S}{ [v]_S}v_S
+\sum_{S\in {\S'_4}} \frac{[h]_S}{ [v]_S}v_S.
\end{align}

We have by multilinearity and the testing assumption
 \eqref{testing}
$$|\Lambda_{R}([h]_Rv,\hat{\mathfrak g})|\le 
C F (|R|/|Q|)^{1-1/q-1/r} \|\hat{\mathfrak g}_R\|_{q} \|v\|_r \le
CF|R||Q|^{-1} \|g\|_{q} \|h\|_r,$$ 
where we have used \eqref{uppertnorm2}, \eqref{eq:9}, norm bound
condition \eqref{normbound} for $v$ and $\|h\|_r=|Q|^{1/r}$.

Next we consider the sum from \eqref{eq:23} involving the stopping cubes in
$\S$ and we write with  the smoothness condition \eqref{smoothness}
\[
\Lambda_{R}\Big(\sum_{S\in \S} h_S -[h]_S b_{\tau,\Q_S,k}, \hat{\mathfrak g}\Big)
\]
\begin{equation}
\label{eq:12}
=\sideset{}{'}\sum_{S\in \S}\Lambda_{S}( h -[h]_S b_{\tau,\Q_S,k}, \hat{\mathfrak g})
+\sideset{}{''}\sum_{S\in \S
}\Lambda_{S}(h -[h]_S b_{\tau,\Q_S,k}, \hat{\mathfrak g}),
\end{equation}
where $\sum'$ denotes summation restricted to the cubes 
which are contained in cubes of the first stopping time and
$\sum''$ denotes summation restricted to the cubes
which are not contained in any cube of the first stopping time. 

 Let
us handle the first summand of \eqref{eq:12} and note if $S\subseteq
P$ for some $P\in\P$ then there exists $\varphi_S$ such that
$\hat{\mathfrak g}_S=\varphi_Su_S$ with the additional provision that
$\varphi_S=0$ if $u_S=0$. Indeed, if $S$ is contained in a stopping
cube $P$ of the first stopping time, then those terms in
\eqref{ftildeexpansion} which are not multiple of $u$ cancel and we
have in case $P\in \P\setminus \P_4$
\begin{align}
  \label{eq:17}
  |\varphi_S|
\le |[g]_{Q}|
+
\bigg|\frac{ [g]_{P}}{ [u]_{P}}\bigg|
+
\sum_{P'\in \overline{\P}_4\cup {\P'_4}: S\subseteq P'}
\bigg|\frac{ [g]_{P'}}{ [u]_{P'}}\bigg|
\le C|Q|^{-1/q}\|g\|_{q}.
\end{align}
Similarly  if $P\in \P\cap \P_4$
\begin{align}
  \label{eq:18}
  |\varphi_S|
\le |[g]_{Q}|+\sum_{P'\in \overline{\P}_4: S\subseteq P'}
\bigg|\frac{ [g]_{P'}}{ [u]_{P'}}\bigg|
\le C|Q|^{-1/q}\|g\|_{q}.
\end{align}
To estimate the first sum in \eqref{eq:12} we use \eqref{eq:16} to
conclude that
\[
\sideset{}{'}\sum_{S\in \S}|\Lambda_{S}(h -[h]_S b_{\tau,\Q_S,k}, \varphi_Su)|\le CF\|g\|_{q}
\sum_{S\in \S}(|S|/|Q|)^{1-1/r} \|h_S -[h]_S b_{\tau,\Q_S,k}\|_{r} 
\] 
\[
\le CF|R||Q|^{-1}\|g\|_{q}\|h\|_{r}.
\]
 To handle the second
summand in \eqref{eq:12} observe that
\[
\sideset{}{''}\sum_{S\in \S}
\Lambda_{S}( h -[h]_S b_{\tau,\Q_S,k}, \hat{\mathfrak g})
\]
\begin{align}
  \label{eq:19}
  =\sideset{}{''}\sum_{S\in \S}\Lambda_{S}( h, \mathfrak g-[\mathfrak
g]_S[u]_S^{-1}u)
\end{align}
\begin{align}
  \label{eq:20}
-\sideset{}{''}\sum_{S\in \S}\Lambda_{S}( [h]_S b_{\tau,\Q_S,k}, \mathfrak g-[\mathfrak
g]_S[u]_S^{-1}u)
\end{align}
\begin{align}
  \label{eq:21}
+
  \sideset{}{''}\sum_{S\in \S}\Lambda_{S}\big( h -[h]_S b_{\tau,\Q_S,k}, ([\mathfrak
g]_S[u]_S^{-1}-[\mathfrak
g]_R[u]_R^{-1})u\big).  
\end{align}
We use  estimate
\eqref{multilinearrbound} with $A'$
 and obtain for \eqref{eq:19} that
\[
\sideset{}{''}\sum_{S\in \S}
|\Lambda_{S}( h, \mathfrak g-[\mathfrak
g]_S[u]_S^{-1}u)|\le A'F|Q|^{-1}\|g\|_q\|h\|_r \sum_{S\in \S}|S|
 \le (1-\varepsilon)A'F|R||Q|^{-1}\|g\|_q\|h\|_r.
\]
Here we have used that  $\sum_{S\in\S} |S|\le (1-\varepsilon)|R|$.
In view of testing condition \eqref{testing} with data $\tau, \Q_S$ we see that
\[
\sideset{}{''}\sum_{S\in \S}|\Lambda_{S}( b_{\tau,\Q_S,k}, \mathfrak g-[\mathfrak
g]_S[u]_S^{-1}u)|\le CF\|h\|_{r}
\sideset{}{''}\sum_{S\in \S}(|S|/|Q|)^{1-1/q} \|\mathfrak g_S-[\mathfrak
g]_S[u]_S^{-1}u_S\|_{q}
\]
\[
\le CF\|g\|_q\|h\|_{r}|Q|^{-1}
\sum_{S\in \S}|S| \le CF|R||Q|^{-1}\|g\|_q\|h\|_{r} 
\]
since $S$ is not contained in any cube of the first stopping time and $\|\mathfrak g_S-[\mathfrak
g]_S[u]_S^{-1}u_S\|_{q}\le (|S|/|Q|)^{1/q}\|g\|_q$. This gives the
desired bound for \eqref{eq:20}. To estimate \eqref{eq:21} observe
that
\[
|[\mathfrak
g]_S[u]_S^{-1}-[\mathfrak
g]_R[u]_R^{-1}|\le C \big(|[\mathfrak
g]_S|+|[\mathfrak
g]_R|\big)\le |Q|^{-1/q}\|g\|_q. 
\]
Therefore, with the aid of \eqref{eq:16} we obtain
\[
|Q|^{-1/q}\|g\|_q\sideset{}{''}\sum_{S\in \S}|\Lambda_{S}(h -[h]_S b_{\tau,\Q_S,k}, u)| 
\]
\[
\le CF\|g\|_q\sideset{}{''}\sum_{S\in \S}(|S|/|Q|)^{1-1/r} \| h_S -[h]_S b_{\tau,\Q_S,k}\|_{r}
\le CF|R||Q|^{-1}\|g\|_q\|h\|_r.
\]

We now consider \eqref{eq:24} and note that by smoothness condition
\eqref{smoothness} we have  
\[
\sum_{S\in \S\setminus \S_4} \Lambda_R\Big([h]_Sb_{\tau, \Q_S,
  k}-\frac{[h]_S}{ [v]_S}v_S, \hat{\mathfrak g}\Big)
\]
\begin{align}
  \label{eq:26}
\Big(\sideset{}{'}\sum_{S\in \S\setminus \S_4}+
\sideset{}{''}\sum_{S\in \S\setminus \S_4}\Big) \Lambda_S\Big([h]_Sb_{\tau, \Q_S,
  k}-\frac{[h]_S}{ [v]_S}v, \hat{\mathfrak g}\Big).
\end{align}
For the first sum in \eqref{eq:26}, in view of \eqref{eq:16}, we
obtain the desired bound, since
\[
\sideset{}{'}\sum_{S\in \S\setminus \S_4}\Lambda_S\Big([h]_Sb_{\tau, \Q_S,
  k}-\frac{[h]_S}{ [v]_S}v, \hat{\mathfrak g}\Big)=
\sideset{}{'}\sum_{S\in \S\setminus \S_4}\varphi_S
\Lambda_S\Big([h]_Sb_{\tau, \Q_S,
  k}-\frac{[h]_S}{ [v]_S}v, u\Big)
\]
and $|\varphi_S|\le |Q|^{-1/q}\|g\|_q$.
To estimate the second sum in \eqref{eq:26} we have to proceed in a
similar way as for the second sum from \eqref{eq:12}. Namely, we write
\[
\sideset{}{''}\sum_{S\in \S\setminus \S_4}\Lambda_S\Big([h]_Sb_{\tau, \Q_S,
  k}-\frac{[h]_S}{ [v]_S}v, \hat{\mathfrak g}\Big)
\]
\begin{align}
  \label{eq:27}
  =\sideset{}{''}\sum_{S\in \S\setminus \S_4}[h]_S\Lambda_S(b_{\tau, \Q_S,
  k}, \mathfrak g-[\mathfrak g]_S[u]^{-1}_Su)
\end{align}
\begin{align}
  \label{eq:28}
  -\sideset{}{''}\sum_{S\in \S\setminus \S_4}\frac{[h]_S}{ [v]_S}\Lambda_S(v, \mathfrak g-[\mathfrak g]_S[u]^{-1}_Su)
\end{align}
\begin{align}
  \label{eq:29}
+\sideset{}{''}\sum_{S\in \S\setminus \S_4}
  \Lambda_S\Big([h]_Sb_{\tau, \Q_S,
  k}-\frac{[h]_S}{ [v]_S}v, ([\mathfrak g]_S[u]^{-1}_S-[\mathfrak g]_R[u]^{-1}_R)u\Big).
\end{align}
The sum in \eqref{eq:27} can be estimated by the testing condition
\eqref{testing} with the data $\tau, \Q_S$. The sum in \eqref{eq:28}
can be estimated in view of \eqref{upperdualt2} since $(\mathfrak
g-[\mathfrak g]_S[u]^{-1}_Su)\mathds{1}_S$ has mean zero. Arguing
similarly as in the proof of \eqref{eq:21} we can estimate the sum in \eqref{eq:29}.

Finally, it remains to bound  \eqref{eq:25} which can be written as $\sum_{S\in
  \overline{\S}_4} \xi^S$, where for each $S\in \overline{\S}_4$ the
function $\xi^S$ is defined as
\begin{equation}
\label{eq:30}
- \frac{[h]_S }{ [v]_S}v_S
+\sum_{S'\in {\S'_4}:\overline{S'}=S} \frac{[h]_{S'}}{ [v]_{S'}}v_{S'}
+\sum_{S'\in\S\cap\S_4: \overline{S'}= S}[h]_{S'} b_{\tau,\Q_{S'},k}.
\end{equation}
Exactly $S$ and the children of $S$ contribute to $\xi^S$. Note that
at least one child of $S$ is in $\S\cap \S_4$. The mean of $\xi^S$ is
zero. Hence we can write with the smoothness condition
\eqref{smoothness}
$$\sum_{S\in \overline{\S}_4}\Lambda_{R}(h, \xi^S)=
\sum_{S\in \overline{\S}_4}\Lambda_S(h, \xi^S).$$
Expanding $\xi^S$ again into three terms as in \eqref{eq:30} and considering the terms 
separately, we obtain 
\begin{equation}
\label{eq:31}
\sum_{S\in \S\cap \S_4}
\Lambda_{\overline{S}}( 
{[h]_S}b_{\tau, {\Q_S}, k}, \hat{\mathfrak g})
- \sum_{S\in
  \overline{\S}_4}
\Lambda_S\Big( 
\frac{[h]_S}{ [v]_S}v, \hat{\mathfrak g}\Big)
+ \sum_{S\in {\S'_4}}
\Lambda_{\overline{S}}\Big(
\frac{[h]_S}{ [v]_S}v, \hat{\mathfrak g}\Big).
\end{equation}
Similarly as  above we have to split the sums into $\sum'$ and
$\sum''$. However, in the first and the third sum in \eqref{eq:31}
$\sum'$ denotes summation restricted to the cubes whose parents are contained
in cubes of the first stopping time and $\sum''$ denotes summation
restricted to the cubes whose parents are not contained in any cube of the
first stopping time.

Then we write for each $1\le j\le n$
$$f_j \mathds{1}_{\overline{S}}=f_j \mathds{1}_{S}+\sum_{S'\neq S, \overline{S'}=\overline{S}} f_j \mathds{1}_{S'}$$
and expand the multilinear form correspondingly.
Any term in the expansion which has a $f_j\mathds{1}_{S
'}$ for  some $j$ can be estimated
by the decay condition \eqref{decay} so that we obtain for the penultimate
display the upper bound
$$\Big|\sum_{S\in \S\cap \S_4}\Lambda_{S}([h]_S b_{\tau,\Q_S,k}, \hat{\mathfrak g})\Big|
+C \sum_{S\in \S\cap \S_4} \prod_{\atop{j=1}{j\not=k+1}}^n \|f_j\mathds{1}_{\overline{S}}\|_{p_j}\|\hat{\mathfrak g}_{\overline{S}}\|_q
\le C F\|g\|_q\|h\|_r.$$
Here we have estimated the first term using testing condition
\eqref{testing} with the data $\tau, \Q_S$.

To estimate the  second and third sum in \eqref{eq:31} we add and
subtract $[h]_S b_{\tau,\Q_S,k}$. The sums involving $[h]_S
b_{\tau,\Q_S,k}$ can be estimated in the same way as the first sum in
\eqref{eq:31}. The sums involving 
$[h]_S[v]_S^{-1}v_S-[h]_S b_{\tau,\Q_S,k}$  are estimated similarly to
the case in \eqref{eq:26}. The cubes in
$\overline{\S}_4$ are not pairwise disjoint, but they have bounded
overlap since they are not contained in any cube of $\S_5$ by
construction.  Similarly the cubes in ${\S'_4}$ have bounded
overlap. This completes the proof of \eqref{eq:2}. We are thus reduced to showing \eqref{tildehatbound}.

\subsection{The main estimate}

For a dyadic cube $T\subseteq R$ define the number $\varphi_T$ by 
$$[\hat{\mathfrak g}]_{T}=\varphi_{T} [u]_{T}$$ 
provided $[u]_{T}\neq 0$. If $[u]_{T}= 0$, then necessarily 
$T$ is contained in a stopping cube $P$  of the first stopping time and 
$\hat{\mathfrak g}$ is a multiple of $u$ on $T$ so that we may define
the number $\varphi_{T}$ by
$$\hat{\mathfrak g}_T=\varphi_{T} u_T,$$
and $\varphi_{T}=0$ if $u$ vanishes on $T$.
If $T$ is not contained in a stopping cube $P$ of the first stopping time,
then we obtain an estimate for $\varphi_T$ by expanding $\mathfrak g$ 
as in \eqref{ftildeexpansion} and noting that for stopping time cubes 
$P$ of $\P$ which intersect $T$ and therefore are strictly contained in $T$
the sum of terms in the expansion of $\mathfrak g$ relating to $P$ has vanishing mean
on $T$: 
$$|\varphi_T|
%\le  \bigg|\frac{ [\mathfrak g]_{T}}{ [u]_{T}}-\frac{ [\mathfrak g]_{R}}{ [u]_{R}}\bigg|
\le C\bigg|\frac{ [g]_{T}}{ [u]_{T}}\bigg|
+C\bigg|\frac{ [g]_{R}}{ [u]_{R}}\bigg|
\le C|Q|^{-1/q}\|g\|_{q}.$$
If $T$ is contained in a stopping cube $P$ of the first stopping time,
then those terms in \eqref{ftildeexpansion} which are
not multiple of $u$ cancel and we have in case $P\in \P\setminus \P_4$
$$|\varphi_T|
\le |[g]_{Q}|
+
\bigg|\frac{ [g]_{P}}{ [u]_{P}}\bigg|
+
\sum_{P'\in \overline{\P}_4\cup {\P'_4}: T\subseteq P'}
\bigg|\frac{ [g]_{P'}}{ [u]_{P'}}\bigg|
\le C|Q|^{-1/q}\|g\|_{q}.$$
Similarly we argue if $P\in \P\cap \P_4$.

Analoguously for a dyadic cube $T\subseteq R$ we define the number $ \psi_T$ by
$$[\mathfrak h]_{T}  = \psi_{T} [v]_{T}
$$
if $[v]_{T}\neq 0$ and by 
$$\mathfrak h_T  = \psi_{T} v_T$$
if $[v]_{T}= 0$ with the additional provision $\psi_T=0$ if
$v_T= 0$.
Similarly as for $\varphi_T$ above we conclude 
$$|\psi_{T}|\le C |R|^{-1/r}\|h_R\|_{r}.$$

Let $N$ be an integer such that the integral kernel of $\Lambda$ 
is constant on all dyadic cubes of length $2^{-N}\ell(R)$, where
$\ell(R)$ denotes the side-length of the cube $R$. It is no harm to
assume all other functions involved are also constant on dyadic cubes
of side-length $2^{-N}\ell(R)$, this can be seen by appropriate limiting process
as $N\to \infty$, none of the estimates below will depend on the specific
choices of $N$.
We write the left-hand side of \eqref{tildehatbound} as

\begin{equation}\label{miniintervals}
\sum_{|T|,|U|=2^{-N}\ell(R)}\Lambda_R( 
 \psi_{T}v_T,
 \varphi_{U}u_{U})
=\sum_{|T|,|U|=2^{1-N}\ell(R)}
\Lambda_R(\psi_{T} 
v_{T} 
,
\varphi_{U}
u_{U})
\end{equation}
$$
-\sum_{|T|,|U|=2^{1-N}\ell(R)}
\Lambda_R\Big( 
\psi_{T} v_T - \sum_{\overline{T'}=T}
\psi_{T'} v_{T'}
,
\varphi_{U} u_{U}\Big)
-\sum_{|T|,|U|=2^{1-N}\ell(R)}
\Lambda_R\Big( 
 \psi_{T}v_T 
,
 \varphi_{U} u_{U}
- \sum_{\overline{U'}=U} \varphi_{U'}u_{U'}
\Big)
$$
$$
+
\sum_{|T|,|U|=2^{1-N}\ell(R)}
\Lambda_R\Big( 
 \psi_{T}v_T 
- \sum_{\overline{T'}=T}\psi_{T'}v_{T'}
,
 \varphi_{U}u_{U}
- \sum_{\overline{U'}=U}\varphi_{U'}u_{U'}\Big).
$$
The function
$$
\psi_{T} v_T - \sum_{\overline{T'}=T}
\psi_{T'} v_{T'}$$
has vanishing mean and is supported on $T$. Hence the smoothness condition \eqref{smoothness} 
turns the second term in the telescoping sum into the diagonal sum
$$
\sum_{|T|=2^{1-N}\ell(R)}
\Lambda_R\Big( 
\psi_{T} v_T - \sum_{\overline{T'}=T}
\psi_{T'} v_{T'}
,
\varphi_{T}u_T\Big).
$$
Similarly the third and fourth term of the telescoping expansion 
turn into diagonal sums. Now we iterate the above telescoping argument.
Since $\mathfrak h$ is supported on $R$ and has integral zero, we may restrict 
the sum to cubes contained in $R$. 
We thus obtain for \eqref{miniintervals}  
\begin{equation}\label{bparaproducts1}
\Lambda_R( \psi_{R} v , \varphi_{R} u) 
\end{equation}
\begin{equation}\label{bparaproducts2}
-\sum_{T\subseteq R}
\Lambda_R\Big( 
 \psi_{T}v_T 
- \sum_{\overline{T'}=T}
 \psi_{T'}v_{T'}
,
  \varphi_{T}
u_{T}\Big)
\end{equation}
\begin{equation}\label{bparaproducts3}
-\sum_{T\subseteq R}
\Lambda_R\Big( 
   \psi_{T}v_T 
,
   \varphi_{T}u_{T}
-   \sum_{\overline{T'}=T}\varphi_{T'}u_{T'}
\Big)
\end{equation}
\begin{equation}\label{bparaproducts4}
+\sum_{T\subseteq R}
\Lambda_R\Big( 
   \psi_{T}v_T 
- \sum_{\overline{T'}=T}
 \psi_{T'}v_{T'}
,
   \varphi_{T}u_{T}
-   \sum_{\overline{T'}=T}\varphi_{T'}u_{T'}
\Big).
\end{equation}
The first term \eqref{bparaproducts1} is estimated by
the testing assumption  \eqref{testing} for the data $\tau, \Q_R$ 
and by the stopping time conditions using that $R$ is not contained
in a stopping time cube of the first stopping time. We obtain that
$$|\Lambda_R( \psi_{R} v , \varphi_{R} u)|
\le C\prod_{j=1}^n\|f_j\mathds{1}_{R}\|_{p_j}
\le C|R||Q|^{-1}\prod_{j=1}^n \|f_j\|_{p_j}.$$
The other terms will be estimated in the next few sections.

\subsection{The estimate for term \eqref{bparaproducts2}}

We write for \eqref{bparaproducts2}
$$
\sum_{T\subseteq R} \Lambda_{R}\Big( \varphi_{T} \big( \psi_{T}v_T -
\sum_{\overline{T'}=T} \psi_{T'}v_{T'}\big) , u\Big)
$$
$$
= \Lambda_{R}\Big( \sum_{T\subseteq R} \varphi_{T}v \big(
\psi_{T}\mathds{1}_T - \sum_{\overline{T'}=T}
\psi_{T'}\mathds{1}_{T'}\big) , u\Big).
$$
In the first step, we have moved the factor $\varphi_{T}$ to the first
entry by bilinearity, and we abandoned the factor $\mathds{1}_T$ in
the second entry thanks to the smoothness condition
\eqref{smoothness}.

Let us define
$$\theta^T=\varphi_{T}
\big( \psi_{T} \mathds{1}_T - \sum_{\overline{T'}=T} \psi_{T'}
\mathds{1}_{T'}\big) $$ and
$$\theta=v \sum_{T\subseteq R} \theta^T.$$
Note that $\theta$ has mean zero. 
By estimate \eqref{upperdualthreshold} from the first stopping time 
we may estimate \eqref{bparaproducts2} by
$$ CF  
 (|R|/|Q|)^{1-1/r} |Q|^{1/q} \| \theta \|_{r}. $$
It remains to show
$$\| \theta\|_{r} \le C|R|^{1/r}|Q|^{-1/r-1/q}\|g\|_{q}\|h\|_{r}
.$$

We may restrict the sum to those $T\subseteq R$ not contained in a
stopping cube of the second stopping time, since the contribution from
cubes $T$ contained in a stopping cube of the second stopping time
vanishes due to the fact $v\theta^T=0$.

The set of such cubes we write as $\T\cup \overline{\P}$, where
$\overline{\P}$ contains those cubes which are parents of stopping
cubes of the second stopping time other than $\overline{\P}_4$, and $\T$ contains
all other cubes, which then are not contained in any stopping cube or
parent in $\overline{\P}$ of any stopping cube of the second stopping
time other than $\overline{\P}_4$.

Let $\P_R$ be the partition of $R$ consisting of all stopping cubes of
the second stopping time, all of which have length at least $2^{-N}\ell(R)$,
and the collection of cubes of side-length $2^{-N}\ell(R)$ not contained in
any of the stopping cubes of the second stopping time.

Let $\tilde{v}$ be such that for every $P\in \P_R$
the function $\tilde{v}$ is constant on $P$ and 
$$\int_P \tilde{v} = \int_P |v|^{r}.$$
Then $\tilde{v}\in L^{\infty}$ and
$$|R| \|\tilde{v}\|_\infty\le C\|v\|_{r}^{r}$$
by the second stopping time construction.  Now splitting the norm
according to the partition $\P_R$ we have
$$\|\theta\|_r^r=\Big\| 
v \sum_{T\in \T\cup \overline{\P}} 
 \theta^T\Big\|_{r}^{r}=\sum_{P\in\P_R}\int_P \tilde{v} 
\Big|\sum_{T\in \T\cup \overline{\P}} 
\theta^T\Big|^{r}\le C \sum_{P\in\P_R}\|\tilde{v}_P\|_\infty  \int_P 
\Big|\sum_{T\in \T\cup \overline{\P}} 
\theta^T\Big|^{r}.$$
We estimate the contributions of $\T$ and $\overline{\P}$ separately. We have
$$\Big\|\sum_{T\in \overline{\P}} \theta^T\Big
\|_{r}\le |R|^{1/r}
\Big\|\sum_{T\in \overline{\P}} \theta^T\Big\|_\infty \le
C |R|^{1/r}\sup_{T\subseteq R}|\varphi_T|\sup_{T\subseteq R}|\psi_T|$$
$$\le C|R|^{1/r} |Q|^{-1/r-1/q}\|g\|_{q}\|h\|_{r}.$$
This is the desired estimate for the $\overline{\P}$-portion of the sum.

Now let  $\hat{v}$ be such that for every $P\in \P_R$
the function $\hat{v}$ is constant on $P$ and 
$$[\hat{v}]_P =[v]_P .$$
Then $\hat{v}\in L^{\infty}$ and  
$$|R| \|\hat{v}\|^{r}_\infty\le C\|v\|_{r}^{r}$$ 
by the stopping time construction. Observe that

$$
\sum_{T\in \T} 
 \theta^T=
\sum_{T\in \T} 
\frac{[\hat{\mathfrak g}]_{T}}{ [u]_{T}}\bigg(
 \frac{[\mathfrak h]_T }{[\hat{v}]_T}\mathds{1}_T - \sum_{\overline{T'}=T}
\frac{[\mathfrak h]_{T'}}{ [\hat{v}]_{T'}}\mathds{1}_{T'}\bigg)
.$$
We expand
\begin{equation}\label{simpleat}
 \frac{[\mathfrak h]_T }{[\hat{v}]_T} - 
\frac{[\mathfrak h]_{T'}}{ [\hat{v}]_{T'}}
=
 \frac{[\mathfrak h]_T -[\mathfrak h]_{T'}}{[\hat{v}]_T} 
+ \frac{[\mathfrak h]_{T'} ([\hat{v}]_{T'}-[\hat{v}]_T)}
{[\hat{v}]_T [\hat{v}]_{T'}}.
\end{equation}
Then for any function $w\in L^{r'}$ such that $\|w\|_{r'}\le 1$ one has
\begin{equation}
  \label{eq:5}
\bigg|\int \sum_{T\in \T} 
\frac{[\hat{\mathfrak g}]_{T}}{ [u]_{T}}\bigg(
 \frac{[\mathfrak h]_T }{[\hat{v}]_T}\mathds{1}_T - \sum_{\overline{T'}=T}
\frac{[\mathfrak h]_{T'}}{ [\hat{v}]_{T'}}\mathds{1}_{T'}\bigg)
w\bigg|  
\end{equation}
\[
=\bigg|\int \sum_{T\in \T} 
\frac{[\hat{\mathfrak g}]_{T}}{ [u]_{T}}\bigg(
 \frac{[\mathfrak h]_T }{[\hat{v}]_T}\mathds{1}_T - \sum_{\overline{T'}=T}
 \frac{[\mathfrak h]_{T'}}{ [\hat{v}]_{T'}}\mathds{1}_{T'}\bigg)
\hat{v}\bigg(
 \frac{[w]_T }{[\hat{v}]_T}\mathds{1}_T - \sum_{\overline{T'}=T}
\frac{[w]_{T'}}{ [\hat{v}]_{T'}}\mathds{1}_{T'}\bigg)\bigg|
\]
\[
\le C\|g\|_{q}|Q|^{-1/q} \sum_{T\subseteq R} \Big(\Delta_T\mathfrak h
\Delta_Tw+E_Tw
\Delta_T\mathfrak h\Delta_T\hat{v} +E_T\mathfrak h \Delta_T\hat{v}\Delta_Tw+E_T\mathfrak h E_Tw(\Delta_T\hat{v})^2\Big)|T|,
\]
where
\begin{equation}
  \label{eq:6}
E_Tf=|[f]_T|\quad \text{and}\quad
\Delta_Tf=\Big(|T|^{-1}\sum_{\overline{T'}=T}|[f]_{T'}-[f]_T|^2|T'|\Big)^{1/2}.  
\end{equation}
Applying Lemma \ref{firststandardlemma} we conclude in view of \eqref{eq:5}
that 
$$\Big\|\sum_{T\in \T} 
 \theta^T\Big\|_r\le C\|g\|_{q}|Q|^{-1/q}\|\mathfrak h\|_{r}\le
C|Q|^{-1/q-1/r}|R|^{1/r}\|g\|_{q}\|h\|_{r}.
$$
This is the desired estimate for the $\T$-portion of the sum.

\subsection{Estimate of the term \eqref{bparaproducts3}}
This term is analoguous to the term \eqref{bparaproducts2}.

\subsection{Estimate of the term \eqref{bparaproducts4}}

We consider \eqref{bparaproducts4}. We may assume that
the sum runs only over those cubes $T$ which are not contained in 
any stopping cube of either of the stopping times, or else one of the
entry functions vanishes.

Let $\mathbf R$ be the set of all maximal cubes from $\P\cup\S$.
Let $\S_R$ be the partition of $R$ consisting of all stopping cubes
from $\mathbf R$, all of which have length at least $2^{-N}\ell(R)$,
and the collection of cubes of side-length $2^{-N}\ell(R)$ not contained in
any of the stopping cubes from $\mathbf R$. Let $\hat{u}$ be the function such that for
every $P\in\S_R
$ the function $\hat{u}$ is constant on $P$ and 
\[
[\hat{u}]_P=[u]_P.
\]
Then $\hat{u}\in L^{\infty}$ and $|R|\|\hat{u}\|_{\infty}^q\le
C\|u_R\|_q^q\le C|R||Q|^{-1}\|u\|_q^q$. Analogously,  let $\hat{v}$ be the function such that for
every $P\in\S_R
$ the function $\hat{v}$ is constant on $P$ and 
\[
[\hat{v}]_P=[v]_P.
\]
Then $\hat{v}\in L^{\infty}$ and $|R|\|\hat{v}\|_{\infty}^r\le
C\|v\|_r^r$.

We may replace the form $\Lambda_R$ by $\Lambda_T$ for fixed $T$ by
using \eqref{smoothness}. We then expand \eqref{bparaproducts4}
 by writing each cube $T$ as union over its  $2^{d}$ children 
and apply this in each component, so that we obtain
\begin{equation}\label{twotondsum}
\sum_{T\subseteq R}\sum_{\overline{T}_1=T}\sum_{\overline{T}_2=T}\ldots \sum_{\overline{T}_n=T}
|\Lambda_T( 
   (\psi_{T} - \psi_{T_k})v_{T_k},
   ( \varphi_{T}-   \varphi_{T_{k+1}})u_{T_{k+1}}
)|\end{equation}
here the $j$-th entry for $j\neq k,k+1$ is $f_j\mathds{1}_{T_j}$,
while the $k$-th and $k+1$-th entry are explicitly given.

We split this sum into the off-diagonal terms, that is the terms for
which $T_j\neq T_i$ for at least one pair $(j,i)$
and the remaining $2^d$ diagonal terms.
The off diagonal terms 
are estimated 
via the decay assumption \eqref{decay} by
$$
C |   \psi_{T} - \psi_{T_k}|\|v_{T_k}\|_{r}
|\varphi_{T}
-   \varphi_{T_{k+1}}|\|u_{T_{k+1}}\|_{q}
\prod_{j\neq k,k+1}\|f_j\mathds{1}_T\|_{p_j}.$$
$$\le CF  |   \psi_{T} - \psi_{T_{k}}|
|\varphi_{T}
-   \varphi_{T_{k+1}}| |T| |Q|^{-1+1/q+1/r}.$$
Observe now that by \eqref{simpleat} we have
\begin{equation}
  \label{eq:8}
\sum_{T\subseteq R}\sum_{\overline{T}_k=T}\sum_{\overline{T}_{k+1}=T}|\psi_{T} - \psi_{T_{k}}|
|\varphi_{T}-   \varphi_{T_{k+1}}| |T|  
\end{equation}
\[
\le C|R|^{1-1/q-1/r}\int\sum_{T\subseteq R}\Big(\sum_{\overline{T'}=T}|\psi_{T} -
\psi_{T'}|^2|T'|\Big)^{1/2}
\Big(\sum_{\overline{T'}=T}
|\varphi_{T}-   \varphi_{T'}|^2 |T'|\Big)^{1/2}|T|^{-1}  |R|^{-1+1/q+1/r}\mathds{1}_T.
\]
Let $s=qr/(q+r)$ and note that the last integral can be controlled from above by
\[
\sup_{\|w\|_{L^{s'}}\le1}\int\sum_{T\subseteq R}\Big(\sum_{\overline{T'}=T}|\psi_{T} -
\psi_{T'}|^2|T'|\Big)^{1/2}
\Big(\sum_{\overline{T'}=T}
|\varphi_{T}-   \varphi_{T'}|^2 |T'|\Big)^{1/2}E_Tw
\]
which in turn can be estimated by
\[
\sup_{\|w\|_{L^{s'}}\le1}\sum_{T\subseteq R} \Big(\Delta_T\mathfrak h\Delta_T\hat{\mathfrak g}+
E_T\hat{\mathfrak g}\Delta_T\mathfrak h\Delta_T\hat{u}+
E_T\mathfrak h\Delta_T\hat{v}\Delta_T\hat{\mathfrak g}+
E_T\mathfrak hE_T\hat{\mathfrak g}\Delta_T\hat{v}\Delta_T\hat{u}\Big)|T|E_Tw.
\]

The diagonal terms are parameterized by $T'$ with $\overline{T'}=T$. We may 
replace $\Lambda_T$ by $\Lambda_{T'}$ since all entry functions
are supported on $T'$.  
We estimate   diagonal terms via 
adding and subtracting a term involving the function 
$b_{\tau,\Q_ {T'},k}$:
\begin{equation}\label{diagterm}  (\psi_{T} - \psi_{T'})( \varphi_{T}-   \varphi_{T'})
\Lambda_{T'}(v,u)
\end{equation}
$$=(\psi_{T} - \psi_{T'})( \varphi_{T}-   \varphi_{T'})
\Lambda_{T'}( {[v]_{T'}}{b}_{\tau,\Q_{T'},k},u)
$$
$$+   (\psi_{T} - \psi_{T'})( \varphi_{T}-   \varphi_{T'})
\Lambda_{T'}( v
-{[v]_{T'}}{b}_{\tau,\Q_{T'},k}, u).
$$
The first term on the right-hand side of \eqref{diagterm}
is estimated via the testing
assumption for $b_{\tau,\Q_{T'},k}$ by 
$$\le C 
 |\psi_{T} - \psi_{T'}|| \varphi_{T}-   \varphi_{T'}|
|[v]_{T'}|\|b_{\tau,\Q_{T'},k}\|_{r}\|u_{T'}\|_{q}\prod_{j\neq k,k+1}\|f_j\mathds{1}_{T'}\|_{p_j}
$$
$$\le CF  |\psi_{T} - \psi_{T'}|| \varphi_{T}-   \varphi_{T'}|
|T||Q|^{-1+1/q+1/r}.
$$
The second term on the right-hand side of \eqref{diagterm}
is estimated via the stopping time condition \eqref{upperdualthreshold}
applied to the cube $T$:
$$\le CF
 |\psi_{T} - \psi_{T'}|| \varphi_{T}-   \varphi_{T'}|
(|T|/|Q|)^{1-1/r}\|v - [v]_{T'}{b}_{\tau,\Q_{T'},k}\|_{r}\|u\|_{q}
$$
$$\le CF  |\psi_{T} - \psi_{T'}|| \varphi_{T}-   \varphi_{T'}|
|T||Q|^{-1+1/q+1/r}.
$$
Observe again that by \eqref{simpleat} we have
\begin{equation}
  \label{eq:7}
\sum_{T\subseteq R}\sum_{\overline{T'}=T}|\psi_{T} - \psi_{T'}|
|\varphi_{T}-   \varphi_{T'}| |T|
\end{equation}
\[
=
|R|^{1-1/q-1/r}\int\sum_{T\subseteq R}\Big(\sum_{\overline{T'}=T}|(\psi_{T} - \psi_{T'})\mathds{1}_{T'}|\Big)
\Big(\sum_{\overline{T'}=T}|(\varphi_{T}-   \varphi_{T'})\mathds{1}_{T'}|\Big)|R|^{1/q+1/r-1}\mathds{1}_R.
\]
The last integral can be controlled by
\[
\sup_{\|w\|_{L^{s'}}\le1}\sum_{T\subseteq
  R}\Big(\sum_{\overline{T'}=T}|\psi_{T} - \psi_{T'}|^2|T'|\Big)^{1/2}
\Big(\sum_{\overline{T'}=T}|\varphi_{T}-
\varphi_{T'}|^2|T'|\Big)^{1/2}E_Tw
\]
which in turns can be dominated by
\[
 C\sup_{\|w\|_{L^{s'}}\le1}\sum_{T\subseteq R} \Big(\Delta_T\mathfrak h\Delta_T\hat{\mathfrak g}+
E_T\hat{\mathfrak g}\Delta_T\mathfrak h\Delta_T\hat{u}+
E_T\mathfrak h\Delta_T\hat{v}
\Delta_T\hat{\mathfrak g}+
E_T\mathfrak hE_T\hat{\mathfrak g}\Delta_T\hat{v}\Delta_T\hat{u}\Big)E_Tw|T|.
\]
Collecting the estimates from \eqref{eq:8} and \eqref{eq:7} and applying Lemma \ref{secondstandardlemma}
we can dominate  \eqref{twotondsum} by
$$\le CF(|R|/|Q|)^{1-1/q-1/r}\|\mathfrak h_{R}\|_{r}\|\mathfrak
g_R\|_{q}
\|\hat{u}_R\|_{\infty}\|\hat{v}_R\|_{\infty}$$
$$\le C F|R||Q|^{-1} \|g\|_q\|h\|_r.$$
This completes the estimation of \eqref{bparaproducts4}.

\subsection{Two standard lemmas via outer measures}

This section contains two standard estimates for martingale sums and differences.
Our purpose will be to reprove these estimates using outer measures
techniques, the use of these outer measure techniques in the context of
$L^p$ estimates in harmonic analysis has been initiated in
\cite{DT}. We present only as much of the material from \cite{DT} as necessary to illustrate our proofs, for more details we refer to \cite{DT}. 

Let $X$ be the subset of the set $\mathcal D$ of all dyadic cubes in $\R^d$
consisting of all dyadic cubes of sidelength at least $2^{-N}$ for suitably
large $N$ and contained in a large compact set of $\R^d$ depending on the
truncation parameters of the form $\Lambda$. All that follows will concern
the collection $X$. 
For a dyadic cube $T$ let $\mathcal D(T)$ denote the set of all dyadic
cubes $T'\subseteq T$. As in \cite{DT} let $\mu$ be the  outer measure
on $X$
generated by the function 
\[
\kappa(\mathcal D(T))=|T|.
\]  
To define the outer measure spaces we have to introduce the so-called
size functions. Namely, for any $p\in[1, \infty)$ and a function $F$
on $\mathcal D$ we define
\[
S_p(F)(\mathcal D(T)):=\Big(|T|^{-1}\sum_{Q\in\mathcal D(T)}|F(Q)|^p|Q|\Big)^{1/p}
\] 
and for $p=\infty$ we set
\[
S_{\infty}(F)(\mathcal D(T)):=\sup_{Q\in\mathcal D(T)}|F(Q)|.
\]
For the size function $S$, which is one of the functions $S_p$ or
$S_{\infty}$ we define the space $L^{\infty}(X, \kappa, S)$ which
consists of all functions $F$ on $X$ such that
\[
\|F\|_{L^{\infty}(X, \kappa, S)}:=\sup_{T\in\mathcal D}|S(F)(\mathcal
D(T))|<\infty.
\] 
In order to define the $L^{p}(X, \kappa, S)$ spaces we
need to introduce the superlevel measure $\mu$ as
\[
\mu(S(F)>\lambda):=\inf\{\mu(\mathcal G): \mathcal G\subseteq \mathcal
D \ \text{and} \  S(F\mathds{1}_{\mathcal D\setminus\mathcal G})\le \lambda\};
\]
this specific definition is the crux of the matter of the theory developed
in \cite{DT}.
Then $L^{p}(X, \kappa, S)$ is the set of all functions $F$ such that
\[
\|F\|_{L^{p}(X, \kappa, S)}:=\bigg(p\int_0^{\infty}\lambda^{p-1}\mu(S(F)>\lambda)d\lambda\bigg)^{1/p}<\infty.
\]
Moreover the weak $L^{p, \infty}(X, \kappa, S)$ is the set of all functions $F$ such that
\[
\|F\|_{L^{p}(X, \kappa, S)}:=\Big(\sup_{\lambda>0}\lambda^{p}\mu(S(F)>\lambda)\Big)^{1/p}<\infty.
\]
For a dyadic cube $T$ and a function $f$ on $\R^d$ define
$$E(f)(T)=E_T f := \left|[f]_T \right|,\quad \text{and}\quad 
\Delta(f)(T)=\Delta_T f :=\Big(|T|^{-1}\sum_{\overline{T'}=T} \left|[f]_{T'}
  -[f]_T\right|^2|T'|\Big)^{1/2}.$$
We now prove the following discrete version of Carleson's embedding theorem.
\begin{thm}
  \label{Caremb}
Let $p\in[1, \infty]$ then there exists a constant $C>0$ such that
for every function $f\in L^p$ we have
\begin{align}
  \label{eq:3}
  \|E(f)\|_{L^{p}(X, \kappa, S_{\infty})}\le C\|f\|_{L^p}
\end{align}
and
\begin{align}
  \label{eq:4}
  \|\Delta(f)\|_{L^{p}(X, \kappa, S_{2})}\le C\|f\|_{L^p}.
\end{align}
\end{thm}
\begin{proof}
  In view of the Marcinkiewicz interpolation theorem for the outer
  measure spaces \cite{DT} it suffices to prove estimates \eqref{eq:3} and
  \eqref{eq:4} for $p=1$ and $p=\infty$. For the proof of \eqref{eq:3}
  note that
\[
\|E(f)\|_{L^{\infty}(X, \kappa, S_{\infty})}=\sup_{T\in\mathcal
  D}\sup_{Q\in\mathcal D(T)}|E_Qf|\le \|f\|_{L^{\infty}}.
\] 
For $p=1$ fix $f\in L^1$ and let $\mathcal F$ be the set of all
maximal cubes $Q$ such that $[|f|]_Q>\lambda$, then by the maximality
of the cubes $Q$ we see that
\[
\sum_{Q\in\mathcal F}|Q|\le \frac{\|f\|_1}{\lambda}.
\] 
This immediately implies that
\[
\mu(S_{\infty}(E(f))>\lambda)\le \mu(\mathcal G)\le\sum_{Q\in\mathcal F}|Q|\le\frac{\|f\|_1}{\lambda}
\]
since $S_{\infty}(E(f)\mathds{1}_{\mathcal D\setminus\mathcal G})\le
\lambda$, where $\mathcal G=\bigcup_{Q\in\mathcal F}\mathcal D(Q)$. This completes the proof of \eqref{eq:3} for $p=1$. 

For the proof of \eqref{eq:4} for $p=\infty$ it is easy to see that
\begin{multline*}
  S_2(\Delta(f))(\mathcal D(T))^2=|T|^{-1}\sum_{Q\in \mathcal D(T)}\sum_{\overline{Q'}=Q}\left|[f]_{Q'}
  -[f]_Q\right|^2|Q'|\\
=|T|^{-1}\sum_{Q\in \mathcal D(T)}\int_Q\Big|\sum_{\overline{Q'}=Q}([f]_{Q'}
  -[f]_Q)\mathds{1}_{Q'}\Big|^2\le|T|^{-1}\int_T|f|^2\le \|f\|_{\infty}^2. 
\end{multline*}
In the case $p=1$ we perform the Calder\'on--Zygmund
decomposition at a height $\lambda>0$. Let 
\[
b=\sum_{Q\in\mathcal F}(f-[f]_Q)\mathds{1}_Q
\] 
be a bad function and let $g=f-b$ be a good function. We see that
$[b]_Q=0$ for every $Q\in\mathcal F$ and $\|g\|_{\infty}\le
2^{d}\lambda$. Finally, we obtain that
\[
\mu(S_2(\Delta(f))>2^d\lambda)\le \mu(\mathcal G)\le
\sum_{Q\in\mathcal F}|Q|\le\frac{\|f\|_1}{\lambda}
\]
since $S_2(\Delta(f)\mathds{1}_{\mathcal D\setminus\mathcal
  G})=S_2(\Delta(g)\mathds{1}_{\mathcal D\setminus\mathcal G})\le
2^d\lambda$. This completes the proof of the theorem. 
\end{proof}

These are norm estimates for the martingale average and the martingale
difference of $f$ on $T$. 
\begin{lem}\label{firststandardlemma}
Let $f_1\in L^p$ and $f_2\in L^{p'}$ with $1<p<\infty$ and $f_3\in L^{\infty}$ such that
$\|f_3\|_{\infty}\le 1$. Assume
we are given coefficients $\alpha_T$ such that for every dyadic cube
$T\subseteq R$ we have
\[
|\alpha_T|\le C\Big(
\Delta_Tf_1
\Delta_Tf_2+E_Tf_2
\Delta_Tf_1\Delta_Tf_3 +E_Tf_1 \Delta_Tf_3\Delta_Tf_2+E_Tf_1 E_Tf_2(\Delta_Tf_3)^2\Big)|T|.
\]
Then there exists a constant $C>0$ such that
\[
\sum_{T\subseteq R} |\alpha_T|\le C\|f_1\mathds{1}_R\|_{p}\|f_2\mathds{1}_R\|_{p'}.
\]
\end{lem}
\begin{proof}
  By the H\"older's inequality for the outer measure spaces \cite{DT}
  we see that
\[
\sum_{T\subseteq R}\Delta_Tf_1
\Delta_Tf_2|T|\le \|\Delta_Tf_1\|_{L^{p}(X, \kappa, S_2)}
\|\Delta_Tf_2\|_{L^{p'}(X, \kappa, S_2)},
\]
\[
\sum_{T\subseteq R}E_Tf_1 \Delta_Tf_3\Delta_Tf_2|T|\le \|E_Tf_1\|_{L^{p}(X, \kappa, S_{\infty})}
\|\Delta_Tf_2\|_{L^{p'}(X, \kappa, S_2)}\|\Delta_Tf_3\|_{L^{\infty}(X, \kappa, S_2)}
\]
In a similar way we proceed with $E_Tf_1 \Delta_Tf_3\Delta_Tf_2$ and
$E_Tf_1 E_Tf_2(\Delta_Tf_3)^2$. Finally, applying Theorem \ref{Caremb}
we obtain the desired bounds. 
\end{proof}

\begin{lem}\label{secondstandardlemma}
Let $f_1\in L^p$, $f_2\in L^q$, $f_3\in L^{(pq/(p+q))'}$ with $1<p, q<\infty$ and let $f_4,
f_5\in L^{\infty}$ be  such that
 $\|f_4\|_\infty\le 1$, $\|f_5\|_\infty\le 1$. Assume we are given coefficients
$\alpha_T$ such that for every dyadic cube $T\subseteq R$ we have
\[
|\alpha_T|\le C \Big(\Delta_Tf_2\Delta_Tf_1+
E_Tf_1\Delta_Tf_2\Delta_Tf_5+
E_Tf_2\Delta_Tf_4
\Delta_Tf_1+
E_Tf_2E_Tf_1\Delta_Tf_4\Delta_Tf_5\Big)E_Tf_3|T|.
\]
Then
$$\sum_{T\subseteq R}|\alpha_T|\le
\|f_1\mathds{1}_R\|_{L^p}\|f_2\mathds{1}_R\|_{L^{q}}
\|f_3\mathds{1}_R\|_{L^{(pq/(p+q))'}}.$$
\end{lem}
The proof of this lemma is similar to the proof of the previous lemma.

\end{document}